\documentclass{amsart}

\usepackage{amsmath,amssymb}
\usepackage{amsthm}
\usepackage{amsrefs}
\usepackage{indentfirst}
\usepackage{mathrsfs}
\usepackage{hyperref}
\usepackage{xcolor}
\usepackage[margin=1.4in]{geometry}
\usepackage{nicefrac}
\usepackage{xifthen}
\usepackage{physics}
\usepackage{algorithm}
\usepackage{algorithmicx}
\usepackage{algpseudocode}
\usepackage{caption}
\usepackage{subcaption}
\usepackage{graphics}
\usepackage{tikz}
\usetikzlibrary{positioning}

\newcommand{\bR}{\mathbb{R}}
\newcommand{\calU}{\mathcal{U}}
\newcommand{\tu}[1]{\ifthenelse{\isempty{#1}}{\vec{\mathbf{u}}} {\mathbf{u}^{[#1]}}}
\newcommand{\tT}{\mathbf{T}}
\newcommand{\CPTR}{\mathcal{C}S_{\text{TR}}}
\newcommand{\CPTT}{\mathcal{C}S_{\text{TT}}}
\DeclareMathOperator{\rk}{rank}

\numberwithin{equation}{section}
\newtheorem{theorem}{Theorem}[section]
\newtheorem{lemma}[theorem]{Lemma}
\newtheorem{remark}[theorem]{Remark}
\newtheorem{proposition}[theorem]{Proposition}
\newtheorem{assumption}[theorem]{Assumption}
\newtheorem{definition}[theorem]{Definition}
\newtheorem{corollary}[theorem]{Corollary}
\allowdisplaybreaks[4]

\title[One-dimensional Tensor Network Recovery]{One-dimensional Tensor Network Recovery}
\author{Ziang Chen}
\address{(ZC) Department of Mathematics, Massachusetts Institute of Technology, 77 Massachusetts Avenue, Cambridge, MA 02139.}
\email{ziang@mit.edu}
\author{Jianfeng Lu}
\address{(JL) Departments of Mathematics, Physics, and Chemistry, Duke University, Box 90320, Durham, NC 27708.}
\email{jianfeng@math.duke.edu}
\author{Anru R. Zhang}
\address{(ARZ) Departments of Biostatistics \& Bioinformatics and Computer Science, Duke University, Box 2721, Durham, NC 27710.}
\email{anru.zhang@duke.edu}
\date{\today}

\thanks{The work of ZC and JL is supported in part by the National Science Foundation via awards DMS-2012286 and DMS-2309378. The work of ARZ is supported in part by the National Science Foundation via grant CAREER-2203741 and the National Institutes of Health via grant 1R01HL169347-01A1.}

\begin{document}

\begin{abstract}
We study the recovery of the underlying graphs or permutations for tensors in the tensor ring or tensor train format. Our proposed algorithms compare the matricization ranks after down-sampling, whose complexity is $O(d\log d)$ for $d$-th order tensors. We prove that our algorithms can almost surely recover the correct graph or permutation when tensor entries can be observed without noise. We further establish the robustness of our algorithms against observational noise. The theoretical results are validated by numerical experiments. 
\end{abstract}



\maketitle

\section{Introduction}
Tensors, or multidimensional arrays, are powerful in encoding and computing high-order information and have a wide range of applications in natural sciences and scientific computing. Since the storage and associated computational complexity of a tensor can be exponential in its order, many tensor formats are developed to reduce the storage and computation cost, e.g., the Tucker format \cite{tucker1963implications} and tensor networks \cite{orus2014practical}.

Tensor network models contract low-order tensors with respect to an underlying graph to form high-order tensors. They are widely used in many fields, such as quantum many-body physics \cites{orus2019tensor, orus2014practical, yuan2021quantum, schollwock2011density, verstraete2006matrix, hallberg2006new} and machine learning/data science \cites{stoudenmire2016supervised, cichocki2016tensor, cichocki2017tensor, cheng2019tree, socher2013reasoning, ren2023high}. Among all tensor network structures, the one-dimensional network attracts the most attention due to its practical and efficient numerical performance. Such structure represents high-order tensors using a group of $3$-tensors ordered as a loop or a path, and is known as matrix product states (MPS) \cites{schollwock2011density, orus2014practical} in physics, and tensor ring (TR) \cite{zhao2016tensor} as well as the tensor train (TT) \cite{oseledets2011tensor} in mathematics and numerical analysis, and thus we adopt the name convention in this paper. The other tensor network formats include tree tensor networks \cite{shi2006classical}, projected entangled pair states (PEPS) \cites{orus2014practical, verstraete2004renormalization}, etc. 

Tensor ring/train decomposition \cites{zhao2016tensor, oseledets2011tensor} is the problem to compute the TR/TT representation of given a high-order tensor and has many applications, such as tensor completion \cites{yuan2019tensor, yuan2018higher, bengua2017efficient, wang2017efficient}, tensor SVD \cite{zhou2022optimal}, and complexity reduction of neural network layers \cite{wang2018wide}. This direction has been well investigated from both algorithmic and theoretical aspects; we refer the interested readers to \cites{chen2020tensor, holtz2012alternating, lubich2013dynamical, khoo2021efficient, chen2022learning}.

The performance of tensor network decomposition tasks relies heavily on the choice of the graph. One gets better performance if one employs the right underlying graph for a target tensor in some tensor network format. Thus in fields like machine learning, scientific computing, and theoretical chemistry, several works are devoted to constructing or discovering the correct underlying graph for tensor networks. An agglomerative approach is proposed in \cite{ballani2014tree} to find the suitable tree structure for hierarchical tensor formats. In \cite{li2020evolutionary}, the authors view the underlying graph as binary strings and do searching in the Hamming space. Since then there have been other works on searching the graph structure or topology for tensor networks, see e.g. \cites{hashemizadeh2020adaptive, nie2021adaptive, liu2023adaptively}. More related to our setting, \cite{li2022permutation} studies the tensor network permutation search, i.e., the problem of finding the best one-to-one mapping from the tensor indices to the vertices of a given graph. The algorithm in \cite{li2022permutation} uses local sampling to iteratively minimize some loss functions whose minimizer corresponds to the best permutation in some sense. The result is further improved in \cite{li2023alternating} with less computational cost. In quantum chemistry and physics literature, one approach proposed to find better index ordering or graph topology is to minimize the entanglement, see e.g., \cites{murg2015tree, szalay2015tensor,hikihara2023automatic} and references therein. \cite{dupuy2021inversion} proposes another strategy using the inversion symmetry property of singular values.

In this work, we consider the recovery of the underlying graph for TR and TT format. More explicitly, for a given high-order tensor that is assumed to be of the TR or TT format, we aim to design efficient and reliable algorithms to recover the underlying graph, i.e., the permutation that maps the indices of the given tensor to correct ordering in the underlying loop or path.

We propose polynomial-time complexity algorithms with theoretical guarantees to solve the tensor indices permutation problem for the TR/TT format. To the best of our knowledge, this is the first approach with complete rigorous analysis including clear complexity bounds, almost sure correctness in the noiseless case, and robustness theorems against the observation error.

The rest of this paper will be organized as follows. In Section~\ref{sec:problem}, after introducing tensor notations and tensor train/train format, we define the task of recovering the underlying graph. The algorithms will be described in Section~\ref{sec:alg} and analyzed rigorously in Section~\ref{sec:theory}. Section~\ref{sec:numerics} contains some numerical experiments and the whole paper will be concluded and discussed in Section~\ref{sec:conclude}.

\section{Problem Statement}
\label{sec:problem}
In this section, we first recall the definitions and notations of two one-dimensional tensor networks: tensor ring and tensor train. Then we state the one-dimensional tensor network recovery problem and our goals. 
For the rest of the paper, $S_d$ is the permutation group on $\{1,2,\dots,d\}$; for a matrix $M$ and a positive integer $R$, let $\sigma_R(M)$ be the $R$-th largest singular value of $M$.

\subsection{Tensors} In this paper, tensors are referred to as multi-dimensional arrays. We focus on the real-valued tensors, while the developed results can be directly extended to the complex value cases. Specifically, a $d$-tensor, or a tensor with order $d$, is some $\mathbf{X}\in\bR^{n_1\times n_2\times\cdots\times n_d}$ with entries being $\mathbf{X}(x_1,x_2,\dots,x_d)\in\bR$, $1\leq x_i\leq n_i$, $1\leq i\leq d$. We call $\vec{n}=(n_1,n_2,\dots,n_d)$ the physical or external dimension of the tensor $\mathbf{X}$. Examples of tensors include vectors ($1$-tensors) and matrices ($2$-tensors). We introduce two operations on tensors below.

\emph{Tensor product}: Let $\mathbf{X}\in \bR^{n_1\times n_2\times\cdots\times n_d}$ be a $d$-tensor and let $\mathbf{Y}\in \bR^{m_1\times m_2\times\cdots\times m_{d'}}$ be a $d'$-tensor. The tensor product of $\mathbf{X}$ and $\mathbf{Y}$, denoted as $\mathbf{X}\otimes \mathbf{Y}$, is a $(d+d')$-tensor in $ \bR^{n_1\times n_2\times\cdots\times n_d\times m_1\times m_2\times \cdots \times m_{d'}}$ defined as
\begin{equation*}
    \mathbf{X}\otimes \mathbf{Y}(x_1,x_2\dots,x_d,y_1,y_2,\dots,y_{d'}) = \mathbf{X}(x_1,x_2,\dots,x_d) \mathbf{Y}(y_1,y_2,\dots,y_{d'})
\end{equation*}
for $1\leq x_i\leq n_i$, $1\leq y_{i'}\leq m_j$, $1\leq i\leq d$, and $1\leq i'\leq d'$.

\emph{Contraction}: Let $\mathbf{X}\in \bR^{n_1\times n_2\times\cdots\times n_d}$ be a $d$-tensor and let $\mathbf{Y}\in \bR^{m_1\times m_2\times\cdots\times m_{d'}}$ be a $d'$-tensor. Suppose that $n_i = m_{i'}$ for some $i\in \{1,2,\dots,d\}$ and $i'\in \{1,2,\dots,d'\}$. Then $\mathbf{X}$ and $\mathbf{Y}$ can be contracted on the two indices with the same dimension and the resulting tensor $\mathbf{Z}\in \bR^{n_1\times \cdots \times n_{i-1} \times n_{i+1}\times \cdots \times n_d \times m_1\times \cdots \times m_{i'-1} \times m_{i'+1}\times \cdots \times m_{d'}}$ is a $(d+d'-2)$-tensor defined via: 
\begin{multline*}
    \mathbf{Z}(x_1,\dots, x_{i-1}, x_{i+1},\dots, x_d, y_1,\dots, y_{i'-1}, y_{i'+1},\dots, y_{d'}) \\
    = \sum_{z = 1}^{n_i} \mathbf{X}(x_1,\dots, x_{i-1}, z, x_{i+1},\dots, x_d) \mathbf{Y}(y_1,\dots, y_{i'-1}, z, y_{i'+1},\dots, y_{d'}).
\end{multline*}
The product of two matrices is a specific example of tensor contraction. Contraction can also be defined on a single tensor $\mathbf{X}\in \bR^{n_1\times n_2\times\cdots\times n_d}$ as long as $n_i = n_{i'}$ for some $1\leq i<i'\leq d$, which results $\mathbf{Z}\in \bR^{n_1\times \cdots \times n_{i-1} \times n_{i+1}\times \cdots \times n_{i'-1} \times n_{i'+1}\times \cdots \times n_d}$ with 
\begin{multline*}
    \mathbf{Z}(x_1,\dots, x_{i-1}, x_{i+1},\dots, x_{i'-1}, x_{i'+1}, \dots, x_d) \\
    = \sum_{z = 1}^{n_i} \mathbf{X}(x_1,\dots, x_{i-1}, x_{i+1},\dots, x_{i'-1}, x_{i'+1}, \dots, x_d) .
\end{multline*}
In particular, trace $\text{tr}(\mathbf{X})$ of a square matrix $\mathbf{X}\in\bR^{n\times n}$ is a contraction. 

\subsection{One-dimensional tensor networks} The tensor network model renders a powerful method to parameterize a high-order tensor by contraction operation on a collection of low-order tensors. The readers are referred to \cite{orus2014practical} for discussions on general tensor network models. In this work, we focus on two prominent one-dimensional tensor network models: tensor ring and tensor train.  

\subsubsection{Tensor ring structure} Let $d\in\mathbb{N}_+$ be a fixed order. For the convenience of presentation, all index symbols on the tensor ring are in a sense of mod $d$, e.g., $\tu{d+1}:=\tu{1}$, $r_{d+1} = r_1$, $k_{d+1} = k_1$. A tensor ring (TR) consists of $d$ $3$-tensors connected through contraction, as in Figure~\ref{fig:TR}, 
\begin{figure}[htb!]
	\centering
	\begin{tikzpicture}[
		roundnode/.style={circle, draw = green},
		]
		
		\draw (0,2) node[roundnode] (u1) {$\tu{1}$};
		\draw (2,0) node[roundnode] (u3) {$\tu{3}$};
		\draw (0,-2) node[roundnode] (u5) {$\tu{5}$};
		\draw (-2,0) node[roundnode] (u7) {$\tu{7}$};
		\draw (2/2^0.5,2/2^0.5) node[roundnode] (u2) {$\tu{2}$};
		\draw (2/2^0.5,-2/2^0.5) node[roundnode] (u4) {$\tu{4}$};
		\draw (-2/2^0.5,-2/2^0.5) node[roundnode] (u6) {$\tu{6}$};
		\draw (-2/2^0.5,2/2^0.5) node[roundnode] (u8) {$\tu{8}$};
		
		\draw[blue,-] (u1.east) -- (u2.north west);
		\draw[blue,-] (u2.south east) -- (u3.north);
		\draw[blue,-] (u3.south) -- (u4.north east);
		\draw[blue,-] (u4.south west) -- (u5.east);
		\draw[blue,-] (u5.west) -- (u6.south east);
		\draw[blue,-] (u6.north west) -- (u7.south);
		\draw[blue,-] (u7.north) -- (u8.south west);
		\draw[blue,-] (u8.north east) -- (u1.west);
		\draw[-] (u1.north) -- (0,3);
		\draw[-] (u3.east) -- (3,0);
		\draw[-] (u5.south) -- (0,-3);
		\draw[-] (u7.west) -- (-3,0);
		\draw[-] (u2.north east) -- (3/2^0.5,3/2^0.5);
		\draw[-] (u4.south east) -- (3/2^0.5,-3/2^0.5);
		\draw[-] (u6.south west) -- (-3/2^0.5,-3/2^0.5);
		\draw[-] (u8.north west) -- (-3/2^0.5,3/2^0.5);
		
	\end{tikzpicture}
	\caption{Tensor ring format}
	\label{fig:TR}
\end{figure}
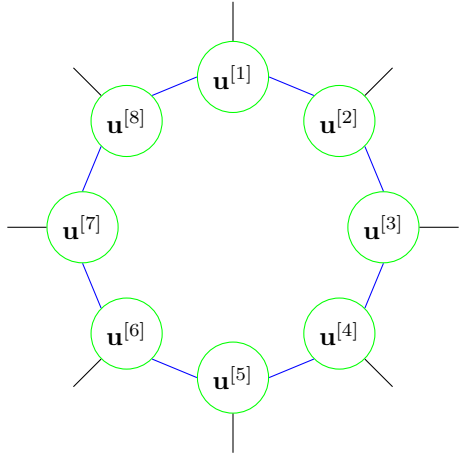
where each node has three edges and represents a $3$-tensor. For any $i\in\{1,2,\dots,d\}$, $\tu{i}$ shares a common edge with $\tu{i+1}$, meaning that a contraction is applied on $\tu{i}$ and $\tu{i+1}$. To define rigorously, given the external/physical dimension $\vec{n}=(n_1,n_2,\dots,n_d)$ and internal/bond dimension $\vec{r}=(r_1,r_2,\dots,r_d)$, any collection of $3$-tensors,
\begin{equation*}
\tu{} = \left(\tu{1},\tu{2}, \dots, \tu{d}\right)\in \calU_{\vec{r},\vec{n}}^d = \bigoplus_{i=1}^d \bR^{{r_i}
	\times n_i \times r_{i+1}},
\end{equation*}
yields a $d$-tensor in TR format:
\begin{equation}\label{eq:TR_entry}
\varphi(\tu{})(x_1,x_2,\dots,x_d)=\tr(\tu{1}(x_1)\tu{2}(x_2)\cdots \tu{d}(x_d)), \quad 1\leq x_i\leq n_i, 1\leq i\leq d,
\end{equation}
where $\tu{i}(x_i):=\tu{i}(:,x_i,:)\in\bR^{r_i \times r_{i+1}}$. Recall that matrix multiplication and trace are both tensor contractions, and \eqref{eq:TR_entry} is well-defined since the third dimension of $\tu{i}$ coincides with the first dimension of $\tu{i+1}$, both equal to $r_{i+1}$, for $i=1,2,\dots,d$. An equivalent way to describe $\varphi(\tu{})$ is via tensor product:
\begin{equation}\label{eq:TR_tensor_prod}
\varphi(\tu{}) = \sum_{1\leq k_i\leq r_i,\ 1\leq i\leq d} \tu{1}_{k_1,k_2}\otimes \tu{2}_{k_2,k_3}\otimes \cdots\otimes \tu{d}_{k_d,k_1},
\end{equation}
where $\tu{i}_{k_i,k_{i+1}} = \tu{i}(k_i,:,k_{i+1})\in \bR^{n_i}$, $1\leq k_i\leq r_i$, $1\leq i\leq d$.

\subsubsection{Tensor train structure} The tensor train (TT) format is a special case of TR format, where $r_1=1$ is fixed and hence the contraction between $\tu{d}$ and $\tu{1}$ is trivial. The TT model is shown in Figure~\ref{fig:TT},
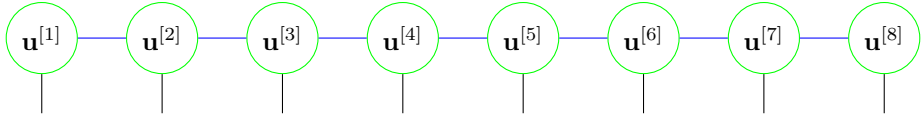
\begin{figure}[htb!]
	\centering
	\begin{tikzpicture}[
	roundnode/.style={circle, draw = green},
	]
	
	\draw (-5.6,0) node[roundnode] (u5) {$\tu{1}$};
	\draw (-4,0) node[roundnode] (u6) {$\tu{2}$};
	\draw (-2.4,0) node[roundnode] (u7) {$\tu{3}$};
	\draw (-0.8,0) node[roundnode] (u8) {$\tu{4}$};
	\draw (0.8,0) node[roundnode] (u1) {$\tu{5}$};
	\draw (2.4,0) node[roundnode] (u2) {$\tu{6}$};
	\draw (4,0) node[roundnode] (u3) {$\tu{7}$};
	\draw (5.6,0) node[roundnode] (u4) {$\tu{8}$};
	
	\draw[blue,-] (u5.east) -- (u6.west);
	\draw[blue,-] (u6.east) -- (u7.west);
	\draw[blue,-] (u7.east) -- (u8.west);
	\draw[blue,-] (u8.east) -- (u1.west);
	\draw[blue,-] (u1.east) -- (u2.west);
	\draw[blue,-] (u2.east) -- (u3.west);
	\draw[blue,-] (u3.east) -- (u4.west);
	\draw[-] (u5.south) -- (-5.6,-1);
	\draw[-] (u6.south) -- (-4,-1);
	\draw[-] (u7.south) -- (-2.4,-1);
	\draw[-] (u8.south) -- (-0.8,-1);
	\draw[-] (u1.south) -- (0.8,-1);
	\draw[-] (u2.south) -- (2.4,-1);
	\draw[-] (u3.south) -- (4,-1);
	\draw[-] (u4.south) -- (5.6,-1);
	
	\end{tikzpicture}
	\caption{Tensor train format}
	\label{fig:TT}
\end{figure}
where there is no edge connecting $\tu{d}$ and $\tu{1}$ due to the trivial contraction. The expression of entries of $\varphi(\tu{})$ \eqref{eq:TR_entry} becomes
\begin{equation*}
\varphi(\tu{})(x_1,x_2,\dots,x_d)=\tu{1}(x_1)\tu{2}(x_2)\cdots \tu{d}(x_d).
\end{equation*}
If we denote $\tu{1}_{k_2} = \tu{1}(1,:,k_2)\in\bR^{n_1}$ for $1\leq k_2\leq r_2$ and $\tu{d}_{k_d} = \tu{d}(k_d,;,1)\in \bR^{n_d}$ for $1\leq k_d\leq r_d$, then \eqref{eq:TR_tensor_prod} can be rewritten as
\begin{equation*}
\varphi(\tu{}) = \sum_{1\leq k_j\leq r_j,\ 2\leq j\leq d} \tu{1}_{k_2}\otimes \tu{2}_{k_2,k_3}\otimes \cdots\otimes \tu{2}_{k_{d-1},k_d}\otimes \tu{d}_{k_d}.
\end{equation*}

\subsection{Recovery of the underlying graph}
In Figures~\ref{fig:TR} and ~\ref{fig:TT}, the construction of TR and TT format requires the prior knowledge of what pairs of $3$-tensors are contracted, or in other words, the underlying graph of the tensor network structure (see the blue parts in Figures~\ref{fig:TR} and \ref{fig:TT}). Note that for TR format, the graph is simply a loop/ring (hence named the tensor ring) with $d$ vertices and $d$ edges, where each vertex corresponds to a $3$-tensor and is associated with two edges. For TT format, the graph is a path of length $d-1$ that visits all $d$ vertices.

However, such prior knowledge is often inaccessible. In this work, we assume that the underlying graph is unknown. Since the graph for TR or TT format is just a loop or a path, all information carried by the graph is the ordering of the vertices. This can be specified as a permutation $\tau\in S_d$ which then fixes the neighbors (vertex with distance $1$) of each vertex. For example, one can choose $\tau=\bigl(\begin{smallmatrix}
1 & 2 & 3 & 4 & 5 & 6 & 7 & 8 \\ 1 & 4 & 2 & 6 & 7 & 5 & 3 & 8\end{smallmatrix} \bigr)$ for the TR format in Figure~\ref{fig:TR_perm} 
\begin{figure}[htb!]
	\centering
	\begin{subfigure}[b]{0.48\textwidth}
		\centering
		\begin{tikzpicture}[
		roundnode/.style={circle, draw = green},
		]
		
		\draw (0,2) node[roundnode] (u1) {$\tu{1}$};
		\draw (2,0) node[roundnode] (u3) {$\tu{3}$};
		\draw (0,-2) node[roundnode] (u5) {$\tu{5}$};
		\draw (-2,0) node[roundnode] (u7) {$\tu{7}$};
		\draw (2/2^0.5,2/2^0.5) node[roundnode] (u2) {$\tu{2}$};
		\draw (2/2^0.5,-2/2^0.5) node[roundnode] (u4) {$\tu{4}$};
		\draw (-2/2^0.5,-2/2^0.5) node[roundnode] (u6) {$\tu{6}$};
		\draw (-2/2^0.5,2/2^0.5) node[roundnode] (u8) {$\tu{8}$};
		
		\draw[blue,-] (u1.south) -- (u4.north west);
		\draw[blue,-] (u4.north) -- (u2.south);
		\draw[blue,-] (u2.south west) -- (u6.north east);
		\draw[blue,-] (u7.south east) -- (u5.north west);
		\draw[blue,-] (u5.north east) -- (u3.south west);
		\draw[blue,-] (u3.north west) -- (u8.south east);
		\draw[blue,-] (u8.north east) -- (u1.west);
		\draw[blue,-] (u7.south) -- (u6.north west);
		\draw[-] (u1.north) -- (0,3);
		\draw[-] (u3.east) -- (3,0);
		\draw[-] (u5.south) -- (0,-3);
		\draw[-] (u7.west) -- (-3,0);
		\draw[-] (u2.north east) -- (3/2^0.5,3/2^0.5);
		\draw[-] (u4.south east) -- (3/2^0.5,-3/2^0.5);
		\draw[-] (u6.south west) -- (-3/2^0.5,-3/2^0.5);
		\draw[-] (u8.north west) -- (-3/2^0.5,3/2^0.5);
		
		\end{tikzpicture}
		\caption{Nodes ordered by physical indices}
	\end{subfigure}
	\hfill
	\begin{subfigure}[b]{0.48\textwidth}
		\centering
		\begin{tikzpicture}[
		roundnode/.style={circle, draw = green},
		]
		
		\draw (0,2) node[roundnode] (u1) {$\tu{1}$};
		\draw (2,0) node[roundnode] (u3) {$\tu{2}$};
		\draw (0,-2) node[roundnode] (u5) {$\tu{7}$};
		\draw (-2,0) node[roundnode] (u7) {$\tu{3}$};
		\draw (2/2^0.5,2/2^0.5) node[roundnode] (u2) {$\tu{4}$};
		\draw (2/2^0.5,-2/2^0.5) node[roundnode] (u4) {$\tu{6}$};
		\draw (-2/2^0.5,-2/2^0.5) node[roundnode] (u6) {$\tu{5}$};
		\draw (-2/2^0.5,2/2^0.5) node[roundnode] (u8) {$\tu{8}$};
		
		\draw[blue,-] (u1.east) -- (u2.north west);
		\draw[blue,-] (u2.south east) -- (u3.north);
		\draw[blue,-] (u3.south) -- (u4.north east);
		\draw[blue,-] (u4.south west) -- (u5.east);
		\draw[blue,-] (u5.west) -- (u6.south east);
		\draw[blue,-] (u6.north west) -- (u7.south);
		\draw[blue,-] (u7.north) -- (u8.south west);
		\draw[blue,-] (u8.north east) -- (u1.west);
		\draw[-] (u1.north) -- (0,3);
		\draw[-] (u3.east) -- (3,0);
		\draw[-] (u5.south) -- (0,-3);
		\draw[-] (u7.west) -- (-3,0);
		\draw[-] (u2.north east) -- (3/2^0.5,3/2^0.5);
		\draw[-] (u4.south east) -- (3/2^0.5,-3/2^0.5);
		\draw[-] (u6.south west) -- (-3/2^0.5,-3/2^0.5);
		\draw[-] (u8.north west) -- (-3/2^0.5,3/2^0.5);
		
		\end{tikzpicture}
		\caption{Nodes ordered by underlying permutation}
	\end{subfigure}
	\caption{Tensor ring format with permutation}
	\label{fig:TR_perm}
\end{figure}
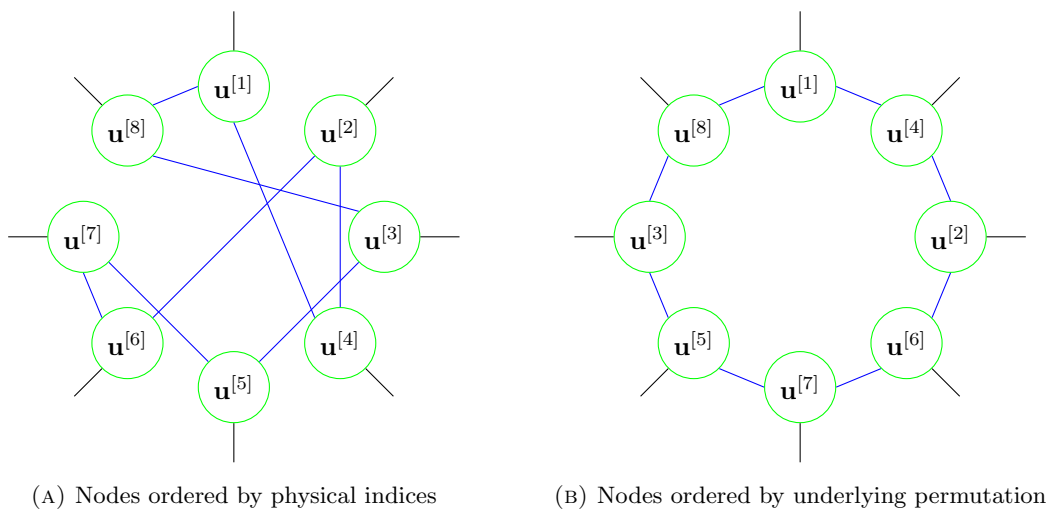
and $\tau = \bigl(\begin{smallmatrix}1 & 2 & 3 & 4 & 5 & 6 & 7 & 8 \\  7 & 5 & 3 & 8 & 1 & 4 & 2 & 6\end{smallmatrix}\bigr)$ for the TT format in Figure~\ref{fig:TT_perm}.
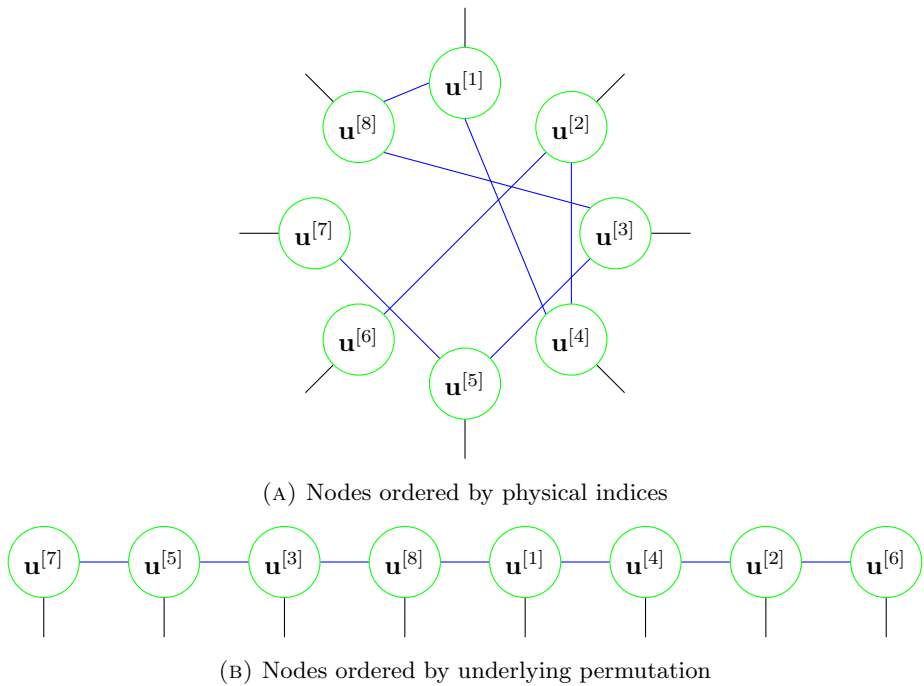
\begin{figure}[htb!]
	\centering
	\begin{subfigure}[b]{0.48\textwidth}
		\centering
		\begin{tikzpicture}[
		roundnode/.style={circle, draw = green},
		]
		
		\draw (0,2) node[roundnode] (u1) {$\tu{1}$};
		\draw (2,0) node[roundnode] (u3) {$\tu{3}$};
		\draw (0,-2) node[roundnode] (u5) {$\tu{5}$};
		\draw (-2,0) node[roundnode] (u7) {$\tu{7}$};
		\draw (2/2^0.5,2/2^0.5) node[roundnode] (u2) {$\tu{2}$};
		\draw (2/2^0.5,-2/2^0.5) node[roundnode] (u4) {$\tu{4}$};
		\draw (-2/2^0.5,-2/2^0.5) node[roundnode] (u6) {$\tu{6}$};
		\draw (-2/2^0.5,2/2^0.5) node[roundnode] (u8) {$\tu{8}$};
		
		\draw[blue,-] (u1.south) -- (u4.north west);
		\draw[blue,-] (u4.north) -- (u2.south);
		\draw[blue,-] (u2.south west) -- (u6.north east);
		\draw[blue,-] (u7.south east) -- (u5.north west);
		\draw[blue,-] (u5.north east) -- (u3.south west);
		\draw[blue,-] (u3.north west) -- (u8.south east);
		\draw[blue,-] (u8.north east) -- (u1.west);
		\draw[-] (u1.north) -- (0,3);
		\draw[-] (u3.east) -- (3,0);
		\draw[-] (u5.south) -- (0,-3);
		\draw[-] (u7.west) -- (-3,0);
		\draw[-] (u2.north east) -- (3/2^0.5,3/2^0.5);
		\draw[-] (u4.south east) -- (3/2^0.5,-3/2^0.5);
		\draw[-] (u6.south west) -- (-3/2^0.5,-3/2^0.5);
		\draw[-] (u8.north west) -- (-3/2^0.5,3/2^0.5);
		
		\end{tikzpicture}
		\caption{Nodes ordered by physical indices}
		\medskip 

	\end{subfigure}
	\begin{subfigure}[b]{\textwidth}
		\centering
		\begin{tikzpicture}[
		roundnode/.style={circle, draw = green},
		]
		
		\draw (-5.6,0) node[roundnode] (u5) {$\tu{7}$};
		\draw (-4,0) node[roundnode] (u6) {$\tu{5}$};
		\draw (-2.4,0) node[roundnode] (u7) {$\tu{3}$};
		\draw (-0.8,0) node[roundnode] (u8) {$\tu{8}$};
		\draw (0.8,0) node[roundnode] (u1) {$\tu{1}$};
		\draw (2.4,0) node[roundnode] (u2) {$\tu{4}$};
		\draw (4,0) node[roundnode] (u3) {$\tu{2}$};
		\draw (5.6,0) node[roundnode] (u4) {$\tu{6}$};
		
		\draw[blue,-] (u5.east) -- (u6.west);
		\draw[blue,-] (u6.east) -- (u7.west);
		\draw[blue,-] (u7.east) -- (u8.west);
		\draw[blue,-] (u8.east) -- (u1.west);
		\draw[blue,-] (u1.east) -- (u2.west);
		\draw[blue,-] (u2.east) -- (u3.west);
		\draw[blue,-] (u3.east) -- (u4.west);
		\draw[-] (u5.south) -- (-5.6,-1);
		\draw[-] (u6.south) -- (-4,-1);
		\draw[-] (u7.south) -- (-2.4,-1);
		\draw[-] (u8.south) -- (-0.8,-1);
		\draw[-] (u1.south) -- (0.8,-1);
		\draw[-] (u2.south) -- (2.4,-1);
		\draw[-] (u3.south) -- (4,-1);
		\draw[-] (u4.south) -- (5.6,-1);
		
		\end{tikzpicture}
		\caption{Nodes ordered by underlying permutation}
	\end{subfigure}
	\caption{Tensor train format}
	\label{fig:TT_perm}
\end{figure}
We introduce the rigorous notations for TR and TT format with permutation as follows.

\subsubsection{Tensor ring structure with permutation} Given permutation $\tau\in S_d$, external dimension $\vec{n}=(n_1,n_2,\dots,n_d)$, and internal dimension $\vec{r}=(r_1,r_2,\dots,r_d)$, denote
\begin{equation*}
    \calU_{\vec{r},\vec{n},\tau}^d = \bigoplus_{i=1}^d \bR^{r_{\tau^{-1}(i)}
    \times n_i \times r_{\tau^{-1}(i)+1}}.
\end{equation*}
Note that the indices/inputs for $\tau$, $r$, etc, are in the sense of $\text{mod }d$, e.g., $\tau(d+1) = \tau(1), r_{d+1} = r_1$. An element in $\calU_{\vec{r},\vec{n},\tau}^d$ can be written as a collection of $3$-tensors:
\begin{equation*}
    \tu{} = \left(\tu{1},\tu{2}, \dots, \tu{d}\right)\in \calU_{\vec{r},\vec{n},\tau}^d,\quad \tu{i}\in \bR^{r_{\tau^{-1}(i)}
    \times n_i \times r_{\tau^{-1}(i)+1}},\ 1\leq i\leq d.
\end{equation*}
Note that the third dimension of $\tu{\tau(j)}$ coincides with the first dimension of $\tu{\tau(j+1)}$ and they both equal to $r_{i+1}$. This corresponds to the fact that in the graph, there is an edge connecting the vertices representing $\tu{\tau(j)}$ and $\tu{\tau(j+1)}$ that can be contracted. 

Denote 
 \begin{equation*}
     \calU_{\vec{r},\vec{n}}^d= \bigcup_{\tau\in S_d} \calU_{\vec{r},\vec{n},\tau}^d\times \{\tau\}
\end{equation*}
that consists of all consistent groups of $3$-tensors as well as the permutation. The map that constructs a $d$-tensor with TR format is given by
\begin{equation*}
    \begin{split}
         \varphi:\ \calU_{\vec{r},\vec{n}}^d\ &\rightarrow \bR^{n_1\times n_2\times \cdots\times n_d}\\
        (\tu{}, \tau) &\mapsto\quad\ \varphi(\tu{}, \tau)
     \end{split}
 \end{equation*}
with entries of $\varphi(\tu{},\tau)$ being
\begin{equation}\label{eq:TR_entry_perm}
    \varphi(\tu{},\tau)(x_1,x_2,\dots,x_d)=\tr(\tu{\tau(1)}(x_{\tau(1)})\tu{\tau(2)}(x_{\tau(2)})\cdots \tu{\tau(d)}(x_{\tau(d)})),
 \end{equation}
 for $1\leq x_i\leq n_i$, $1\leq i\leq d$, where $\tu{i}(x_i):=\tu{i}(:,x_i,:)\in\bR^{r_{\tau^{-1}(i)}\times r_{\tau^{-1}(i)+1}}$. One can also define $\varphi(\tu{},\tau)$ using tensor product:
\begin{equation}\label{eq:TR_tensor_prod_perm}
     \varphi(\tu{},\tau) = \sum_{1\leq k_j\leq r_j,\ 1\leq j\leq d} \tu{\tau(1)}_{k_1,k_2}\otimes \tu{\tau(2)}_{k_2,k_3}\otimes \cdots\otimes \tu{\tau(d)}_{k_d,k_1},
\end{equation}
where $\tu{\tau(j)}_{k_j,k_{j+1}} = \tu{\tau(j)}(k_j,:,k_{j+1})\in \bR^{n_{\tau(j)}}$, $1\leq k_j\leq r_j$, $1\leq j\leq d$, and $k_{d+1}:= k_1$. 
 
One observation is that different permutations may result in the same graph, i.e., the same pairs of adjacent $3$-tensors. More explicitly, let $\alpha,\beta\in S_d$ be defined via $\alpha(j) = j+1$ and $\beta(j) = d+1-j$ for $j=1,2,\dots,d$. Then for any $\tau,\tau'\in S_d$, their underlying graphs are the same if and only if there exists $k\in\{0,1,\dots,d-1\}$ and $\ell\in{0,1}$ such that $\tau' = \tau \circ \alpha^k\circ \beta^\ell$. In other words, one divides $S_d$ into several equivalence classes, and permutations in the same class lead to the same graph. The equivalence class containing $\tau\in S_d$ is 
 \begin{equation}\label{equiv_class_TR}
 \CPTR^d(\tau) = \left\{\tau \circ \alpha^k\circ \beta^\ell : k\in\{0,1,\dots,d-1\},\ \ell\in{0,1}\right\}.
 \end{equation}
 
 \subsubsection{Tensor train structure with permutation} The TT format can still be obtained from the TR format by fixing $r_1=1$, which modified \eqref{eq:TR_entry_perm} into
 \begin{equation*}
     \varphi(\tu{},\tau)(x_1,x_2,\dots,x_d)=\tu{\tau(1)}(x_{\tau(1)})\tu{\tau(2)}(x_{\tau(2)})\cdots \tu{\tau(d)}(x_{\tau(d)}).
 \end{equation*}
 Rewrite \eqref{eq:TR_tensor_prod_perm} as
 \begin{equation*}
     \varphi(\tu{},\tau) = \sum_{1\leq k_j\leq r_j,\ 2\leq j\leq d} \tu{\tau(1)}_{k_2}\otimes \tu{\tau(2)}_{k_2,k_3}\otimes \cdots\otimes \tu{\tau(d-1)}_{k_{d-1},k_d}\otimes\tu{\tau(d)}_{k_d},
 \end{equation*}
where $\tu{\tau(1)}_{k_2} = \tu{\tau(1)}(1,:,k_2)\in\bR^{n_{\tau(1)}}$ for $1\leq k_2\leq r_2$ and $\tu{\tau(d)}_{k_d} = \tu{\tau(d)}(k_d,;,1)\in \bR^{n_{\tau(d)}}$ for $1\leq k_d\leq r_d$.

Similar to the tensor ring case, the permutation corresponding to a graph or path for TT format is not unique. The difference is that the equivalence class containing $\tau$ is
 \begin{equation}\label{equiv_class_TT}
 \CPTT^d(\tau) = \{\tau,\tau\circ\beta\}.
 \end{equation}

\subsubsection{Tasks and goals}
In this work, we study how to infer the unknown underlying graph for TR and TT structures via observations of entries of the whole tensor $\tT$. More specifically, we would like to find all pairs of $3$-tensors for which the associated vertices are adjacent (connected via an edge), that are $(\tu{\tau(j)},\tu{\tau(j+1)})$, $j=1,2,\dots,d$ for TR format, and are $(\tu{\tau(j)},\tu{\tau(j+1)})$, $j=1,2,\dots,d-1$ for TT format.
For simplicity, we say two $3$-tensors are adjacent if their associated vertices are adjacent. Our goal is to design algorithms that are efficient both in sample complexity (number of observed tensor entries) and computational complexity. To sum up, the goal of this paper is: 

\begin{quote}
\emph{Given an oracle that can query (noiseless or noisy) entries of the whole tensor $\tT = \varphi(\tu{},\tau) \in \bR^{n_1\times n_2\times\cdots\times n_d}$ of TR or TT format, where $(\tu{},\tau)\in \calU_{\vec{r},\vec{n}}^d$ with $\tu{}$, $\tau$, and $\vec{r}$ unknown, can we recover the underlying graph, i.e., find some permutation in the equivalence class $\CPTR^d(\tau)$ or $\CPTT^d(\tau)$? If yes, what is the complexity and the probability of success?}
\end{quote}

It is worth remarking that two permutations representing the same tensor $\tT$ do not always belong to the same equivalence class. For example, if $\vec{r} = (1,\ldots,1)$, then $\tT = \varphi(\tu{},\tau)$ can be represented by either TR or TT format with respect to any permutation. To strengthen our argument, we prove in Appendix~\ref{sec:unique_perm} that under certain conditions, only permutations within the same equivalence class can yield the identical tensor $\tT$. This result establishes the uniqueness of the correct permutation with respect to the equivalence relation.

A related question is how to identify a permutation that may not be in the equivalence class $\CPTR^d(\tau)$ or $\CPTT^d(\tau)$, but 
can represent the target tensor $\tT$ with acceptable error or acceptably increased bond dimension. Although this topic of approximate recovery is important, it falls outside the scope of this paper, which concentrates on the precise recovery of the permutation.

\section{Proposed Algorithms}
\label{sec:alg}

In this section, we introduce algorithms for recovering underlying graphs of the tensor ring and tensor train format 
as well as the intuition behind it. The main idea is divide-and-conquer: for determining the correct relative positions of all $d$ indices on a loop (or a path), it suffices to determine the relative positions of any given $4$ indices (for TR format) or $3$ indices (for TT format). Therefore, we first design algorithms for these subproblems and then recover the whole graph based on solutions to the subproblems. The tensor ring and tensor train cases will be covered in Sections~\ref{sec:alg_TR} and \ref{sec:alg_TT}, respectively. 

\subsection{Recovery of the underlying loop for tensor ring format}
\label{sec:alg_TR}

Given four indices in some order, a basic question is whether one could travel around the underlying loop such that the four indices are visited in exactly the same order as given. Note that four is the smallest number of indices such that the question is nontrivial. Thus, one can classify the order into ``correct'' or ``incorrect'' based on the answer to the question. It is clear that the classification is independent of rotations and reflections. Based on this observation, we make the following definitions.

\begin{definition}
For tensor ring format with $\tau\in S_d$ being the underlying permutation and four different indices $i_1,i_2,i_3,i_4\in \{1,2,\dots,d\}$, we say that $(i_1,i_2,i_3,i_4)$ is of the \emph{correct order} with respect to $\tau$ if there exists some $s\in\{1,2,3,4\}$ such that $\tau^{-1}(i_s)<\tau^{-1}(i_{s+1})<\tau^{-1}(i_{s+2})<\tau^{-1}(i_{s+3})$ or $\tau^{-1}(i_s)>\tau^{-1}(i_{s+1})>\tau^{-1}(i_{s+2})>\tau^{-1}(i_{s+3})$, where subscript of $i_\cdot$ is understood via $\text{mod }4$. 
\end{definition}

Figure~\ref{fig:TR_perm} illustrates an example: $(1,2,6,3)$ is of the correct order while $(1,2,3,6)$ is not. For identifying the underlying loop, it suffices to consider whether four indices are of the correct order with respect to $\tau$ due to the following proposition. 

\begin{proposition}\label{prop:equiv_class_TR}
Let $\tau,\tau'\in S_d$. $\tau'\in \CPTR^d(\tau)$ if and only if for any $1\leq j_1<j_2<j_3<j_4\leq d$, $(\tau'(j_1),\tau'(j_2),\tau'(j_3),\tau'(j_4))$ is of the correct order with respect to $\tau$.
\end{proposition}
The proof of \ref{prop:equiv_class_TR} is straightforward and omitted.

Next, we focus on the problem of determining the correct order of four specific indices $i_1,i_2,i_3,i_4\in\{1,2,\dots,d\}$. The main idea is to construct three matrices and compare their ranks or singular values. More specifically, the three matrices are $M_{(i_1,i_2),(i_3,i_4)}\in\bR^{(n_{i_1}n_{i_2})\times (n_{i_3}n_{i_4})}$, $M_{(i_1,i_3),(i_2,i_4)}\in\bR^{(n_{i_1} n_{i_3})\times (n_{i_2} n_{i_4})}$, and $M_{(i_1,i_4),(i_2,i_3)}\in\bR^{(n_{i_1} n_{i_4})\times (n_{i_2} n_{i_3})}$ defined via matricization of the tensor:
\begin{align}\label{M1234_TR}
 & \begin{aligned}
     M_{(i_1,i_2),(i_3,i_4)}& ((x_{i_1},x_{i_2}),(x_{i_3},x_{i_4}))\\
     & = \varphi(\tu{}, \tau)(y_1,\dots,y_{i_1'-1},x_{i_1'}, y_{i_1'+1},\dots,y_{i_2'-1},x_{i_2'}, \\ 
     & \qquad\qquad\ y_{i_2'+1},\dots,y_{i_3'-1},x_{i_3'},y_{i_3'+1},\dots,y_{i_4'-1},x_{i_4'},y_{i_4'+1},\dots,y_d),
 \end{aligned} \\
 & \label{M1324_TR}
     M_{(i_1,i_3),(i_2,i_4)}((x_{i_1},x_{i_3}),(x_{i_2},x_{i_4})) = M_{(i_1,i_2),(i_3,i_4)}((x_{i_1},x_{i_2}),(x_{i_3},x_{i_4})), \\
 \intertext{and} 
 & \label{M1423_TR}
     M_{(i_1,i_4),(i_2,i_3)}((x_{i_1},x_{i_4}),(x_{i_2},x_{i_3})) = M_{(i_1,i_2),(i_3,i_4)}((x_{i_1},x_{i_2}),(x_{i_3},x_{i_4})),
 \end{align}
 for $1\leq x_{i_s}\leq n_{i_s}$, $1\leq s\leq 4$, where $\{i_1',i_2',i_3',i_4'\}=\{i_1,i_2,i_3,i_4\}$, $1\leq i_1'<i_2'<i_3'<i_4'\leq d$, and $y_i\in \{1,2,\dots,n_i\}$ is a fixed index for each $i\in\{1,2,\dots,d\}\backslash\{i_1,i_2,i_3,i_4\}$. These three matrices can be constructed by calling the oracle that returns entries of $\mathbf{T} = \varphi(\tu{}, \tau)$. Intuitively, if $(i_1,i_2,i_3,i_4)$ is of the correct order, $r_1=r_2=\cdots = r_d = R$, and $\vec{n}$ is large enough, then $\rk(M_{(i_1,i_3),(i_2,i_4)})=R^4$, while $\rk(M_{(i_1,i_2),(i_3,i_4)})$ and $\rk(M_{(i_1,i_4),(i_2,i_3)})$ are at most $R^2$, as shown in Figure~\ref{fig:intuition_TR}.
 \begin{figure}[htb!]
    \centering
    \begin{subfigure}[b]{0.48\textwidth}
    \centering
    \begin{tikzpicture}[
        roundnode/.style={circle, draw = green},
    ]
    
    \draw (-1.8,-0.75) node[roundnode] (i1) {$\tu{i_1}$};
    \draw (-1.8,0.75) node[roundnode] (i2) {$\tu{i_2}$};
    \draw (1.8,0.75) node[roundnode] (i3) {$\tu{i_3}$};
    \draw (1.8,-0.75) node[roundnode] (i4) {$\tu{i_4}$};
    
    \draw[blue,-] (i1.north) -- (i2.south);
    \draw[blue,-] (i2.east) -- (i3.west);
    \draw[blue,-] (i3.south) -- (i4.north);
    \draw[blue,-] (i4.west) -- (i1.east);
    \draw[red, dashed] (-2.5,-1.5) rectangle (-1.1,1.5);
    \draw[red, dashed] (1.1,-1.5) rectangle (2.5,1.5);
    \draw[-] (i1.west) -- (-3,-0.75);
    \draw[-] (i2.west) -- (-3,0.75);
    \draw[-] (i3.east) -- (3,0.75);
    \draw[-] (i4.east) -- (3,-0.75);
    \end{tikzpicture}
    \caption{$M_{(i_1,i_2),(i_3,i_4)}$}
    \end{subfigure}
    \hfill
    \begin{subfigure}[b]{0.48\textwidth}
    \centering
    \begin{tikzpicture}[
        roundnode/.style={circle, draw = green},
    ]
    
    \draw (-1.8,-0.75) node[roundnode] (i1) {$\tu{i_1}$};
    \draw (-1.8,0.75) node[roundnode] (i4) {$\tu{i_4}$};
    \draw (1.8,0.75) node[roundnode] (i3) {$\tu{i_3}$};
    \draw (1.8,-0.75) node[roundnode] (i2) {$\tu{i_2}$};
    
    \draw[blue,-] (i1.north) -- (i4.south);
    \draw[blue,-] (i4.east) -- (i3.west);
    \draw[blue,-] (i3.south) -- (i2.north);
    \draw[blue,-] (i2.west) -- (i1.east);
    \draw[red, dashed] (-2.5,-1.5) rectangle (-1.1,1.5);
    \draw[red, dashed] (1.1,-1.5) rectangle (2.5,1.5);
    \draw[-] (i1.west) -- (-3,-0.75);
    \draw[-] (i4.west) -- (-3,0.75);
    \draw[-] (i3.east) -- (3,0.75);
    \draw[-] (i2.east) -- (3,-0.75);
    \end{tikzpicture}
    \caption{$M_{(i_1,i_4),(i_2,i_3)}$}
    \end{subfigure}
    \hfill
    \begin{subfigure}[b]{0.48\textwidth}
    \centering
    \begin{tikzpicture}[
        roundnode/.style={circle, draw = green},
    ]
    
    \draw (-1.8,-0.75) node[roundnode] (i1) {$\tu{i_1}$};
    \draw (-1.8,0.75) node[roundnode] (i3) {$\tu{i_3}$};
    \draw (1.8,0.75) node[roundnode] (i4) {$\tu{i_4}$};
    \draw (1.8,-0.75) node[roundnode] (i2) {$\tu{i_2}$};
    
    \draw[blue,-] (i1.east) -- (i2.west);
    \draw[blue,-] (i2.north west) -- (i3.south east);
    \draw[blue,-] (i3.east) -- (i4.west);
    \draw[blue,-] (i4.south west) -- (i1.north east);
    \draw[red, dashed] (-2.5,-1.5) rectangle (-1.1,1.5);
    \draw[red, dashed] (1.1,-1.5) rectangle (2.5,1.5);
    \draw[-] (i1.west) -- (-3,-0.75);
    \draw[-] (i3.west) -- (-3,0.75);
    \draw[-] (i4.east) -- (3,0.75);
    \draw[-] (i2.east) -- (3,-0.75);
    \end{tikzpicture}
    \caption{$M_{(i_1,i_3),(i_2,i_4)}$}
    \end{subfigure}
    \caption{Intuition of Algorithm~\ref{alg: 4index}: when $(i_1,i_2,i_3,i_4)$ is of the correct order, the rank of $M_{(i_1,i_2),(i_3,i_4)}$ and $M_{(i_1,i_4),(i_2,i_3)}$ is at most $R^2$, while the rank of $M_{(i_1,i_3),(i_2,i_4)}$ is generically at least $R^4$. The reason is that the number of blue lines connecting red boxes is two in (A) and (B), but four in (C), where red boxes group the column indices in the matricization.} 
    \label{fig:intuition_TR}
\end{figure}
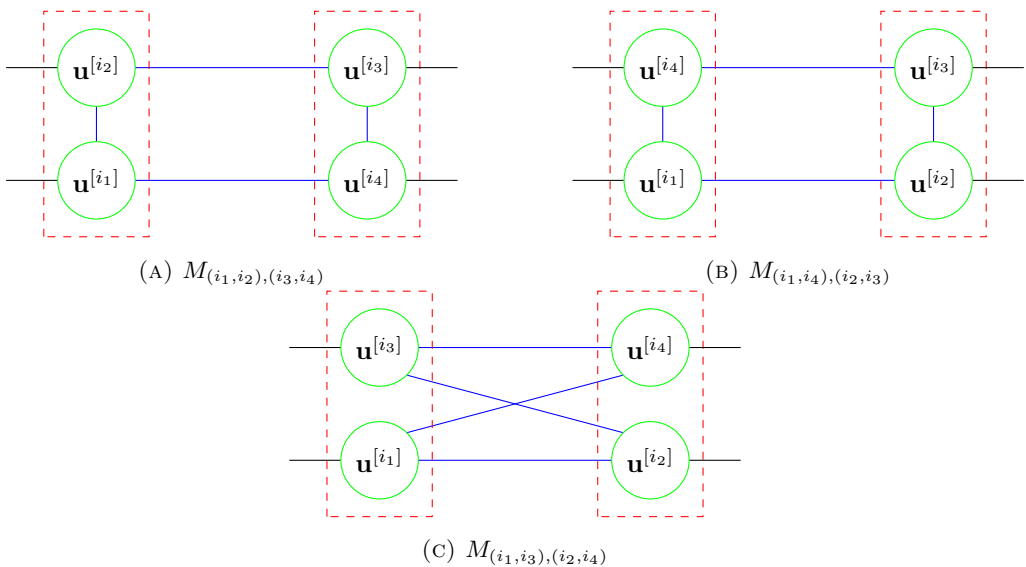
This leads to the design of Algorithm~\ref{alg: 4index}. An input $R$, ideally an approximation of the bond dimension of the tensor ring structure, is required in Algorithm~\ref{alg: 4index}. We will discuss the selection of $R$ later. 

\begin{algorithm}[htb!]
\caption{Determine the order of four indices for TR format}
\label{alg: 4index}
\begin{algorithmic}[1]
\State \textbf{Input:} Four indices $i_1,i_2,i_3,i_4\in\{1,2,\dots,d\}$, $R\in\mathbb{N}_+$ approximating the bond dimension, and an oracle computing entries of $\tT = \varphi(\tu{},\tau)$.
 \State Sample $y_i\in \{1,2,\dots,n_i\}$ for $i\in\{1,2,\dots,d\}\backslash\{i_1,i_2,i_3,i_4\}$.
 \State Construct matrices $M_{(i_1,i_2),(i_3,i_4)}$, $M_{(i_1,i_3),(i_2,i_4)}$, and $M_{(i_1,i_4),(i_2,i_3)}$ via \eqref{M1234_TR}, \eqref{M1324_TR}, and \eqref{M1423_TR}.
 \If {$\sigma_{R^4}(M_{(i_1,i_2),(i_3,i_4)})>\max\{\sigma_{R^4}(M_{(i_1,i_3),(i_2,i_4)}),\sigma_{R^4}(M_{(i_1,i_4),(i_2,i_3)})\}$}
     \State Return $(i_1,i_3,i_2,i_4)$.
 \EndIf
 \If {$\sigma_{R^4}(M_{(i_1,i_3),(i_2,i_4)})>\max\{\sigma_{R^4}(M_{(i_1,i_2),(i_3,i_4)}),\sigma_{R^4}(M_{(i_1,i_4),(i_2,i_3)})\}$}
     \State Return $(i_1,i_2,i_3,i_4)$.
 \EndIf
 \If {$\sigma_{R^4}(M_{(i_1,i_4),(i_2,i_3)})>\max\{\sigma_{R^4}(M_{(i_1,i_2),(i_3,i_4)}),\sigma_{R^4}(M_{(i_1,i_3),(i_2,i_4)})\}$}
     \State Return $(i_1,i_2,i_4,i_3)$.
 \EndIf
 \end{algorithmic}
 \end{algorithm}

 \begin{remark}
 In the noiseless case where entries of $\tT = \varphi(\tu{},\tau)$ can be observed exactly, noticing \eqref{eq:rk_TR}, one can replace the comparison of the $R^4$-th largest singular values of $M_{(i_1,i_2),(i_3,i_4)}$, $M_{(i_1,i_3),(i_2,i_4)}$, and $M_{(i_1,i_4),(i_2,i_3)}$ by the comparison of the ranks of these three matrices. Then Algorithm~\ref{alg: 4index} can be implemented even if the positive integer $R$ approximating the bond dimension is unknown.
 \end{remark}

Now we consider our goal of obtaining some $\tau'\in\CPTR^d(\tau)$. By Proposition~\ref{prop:equiv_class_TR}, it suffices to find some $\tau'\in S_d$, such that $(\tau'(j_1),\tau'(j_2),\tau'(j_3),\tau'(j_4))$ is of the correct order with respect to $\tau$. Note that Algorithm~\ref{alg: 4index} can return the correct order of any four indices that fits the requirement on the desired $\tau'$. So the idea is to apply Algorithm~\ref{alg: 4index} repeatedly and to extend the sequence of indices in the correct order. We make this idea more precise in Definition~\ref{def:consistent_TR}.
 
\begin{definition}\label{def:consistent_TR}
 For TR format, let $\tau\in S_d$ and let $(i_1,i_2,\dots,i_t)$ be a permutation of $t$ different indices in $\{1,2,\dots,d\}$, $4\leq t\leq d$. Then $(i_1,i_2,\dots,i_t)$ is said to be consistent with $\tau$ if $(i_{j_1}, i_{j_2}, i_{j_3}, i_{j_4})$ is of the correct order with respect to $\tau$ for any $1\leq j_1<j_2<j_3<j_4\leq t$.
 \end{definition}
 
Suppose $4\leq t<d$ and we already have a permutation of $(1,2,\dots,t)$, say $(i_1,i_2,\dots,i_t)$, that is consistent with the correct complete order $\tau$. We then aim to insert $(t+1)$ into the right position in $(i_1,i_2,\dots,i_t)$ that maintains the consistency with $\tau$. For any $1\leq j_1<j_2<j_3\leq t$, one can run Algorithm~\ref{alg: 4index} for $t+1$, $i_{j_1}$, $i_{j_2}$, and $i_{j_3}$ to determine where $t+1$ should be inserted: 1) between $i_{j_1}$ and $i_{j_2}$, 2) between $i_{j_2}$ and $i_{j_3}$, or 3) before $i_{j_1}$ or after $i_{j_3}$. Then one can apply the technique similar to binary search to find the correct position of $t+1$ by running Algorithm~\ref{alg: 4index} for at most $O(\log t)$ times and hence obtain a permutation of $(1,2,\dots,t+1)$, say $(i_1,i_2,\dots, i_{t+1})$, that is still consistent with $\tau$. The whole procedure is outlined in Algorithm~\ref{alg: whole_ring}.
 
\begin{algorithm}[htb!]
\caption{Recover the underlying graph for TR format}
\label{alg: whole_ring}
\begin{algorithmic}[1]
\State \textbf{Input:} $R\in\mathbb{N}_+$ approximating the bond dimension and a oracle computing entries of $\tT = \varphi(\tu{},\tau)$.
 \State Let $(i_1,i_2,i_3,i_4)$ be the output of Algorithm~\ref{alg: 4index} with the input four indices being $1,2,3,4$.
 \For {$t = 4 : d-1$}
     \State $j_{\min} = 1$, $j_{\max} = t+1$. \Comment{Use the idea of binary search to find a location in $(i_1,i_2,\dots,i_t)$ such that inserting $t+1$ maintains the consistency with $\tau$.}
     \While {$j_{\max}-j_{\min}\geq 2$}
        \If {$j_{\max}-j_{\min}\geq 3$}
            \State $j_1 = j_{\min}$, $j_2 = j_{\min} + \lfloor\frac{1}{3}\cdot (j_{\max}-j_{\min})\rfloor$, $j_3 = j_{\min} + \lfloor\frac{2}{3}\cdot (j_{\max}-j_{\min})\rfloor$.
        \ElsIf{ $j_{\max} \leq t$}
            \State $j_1 = j_{\min}$, $j_2 = j_{\min} + 1$, $j_3 = j_{\max}$.
        \Else 
            \State $j_1 = 1$, $j_2 = t-1$, $j_3 = t$.
        \EndIf
        \State Run Algorithm~\ref{alg: 4index} with input being $t+1$, $i_{j_1}$, $i_{j_2}$, and $i_{j_3}$.
        \If {the output of Algorithm~\ref{alg: 4index} is equivalent to $(i_{j_1},t+1 , i_{j_2}, i_{j_3})$}
            \State $j_{\min} = j_1$, $j_{\max} = j_2$.
        \ElsIf {the output of Algorithm~\ref{alg: 4index} is equivalent to $(i_{j_1} , i_{j_2}, t+1, i_{j_3})$}
            \State $j_{\min} = j_2$, $j_{\max} = j_3$.
        \Else
            \State $j_{\min} = j_3$.
        \EndIf
    \EndWhile
    \State $(i_1,i_2,\dots,i_{t+1}) = (i_1,i_2,\dots,i_{j_{\min}}, t+1, i_{j_{\max}}, i_{j_{\max}+1},\dots, i_t)$.
 \EndFor
 \State Return $\tau'\in S_d$ with $\tau'(j) = i_j$, $1\leq j\leq d$.
 \end{algorithmic}
 \end{algorithm}
 
\begin{remark}
Algorithm~\ref{alg: whole_ring} combines the partial order of four indices output by Algorithm~\ref{alg: 4index} to recover the order of all $d$ indices. To make the overall procedure more robust, one can run Algorithm~\ref{alg: 4index} several times on each quadruplet ($t+1$, $i_{j_1}$, $i_{j_2}$, $i_{j_3}$), and take the majority vote of the outputs to determine their order. 
\end{remark}

Algorithm~\ref{alg: whole_ring} can be implemented by observing $O(d\log d\cdot n_{\max}^4)$ entries of $\tT=\varphi(\tu{},\tau)$ and computing the $R^4$-th singular value of $O(d\log d)$ matrices of size no larger than $n_{\max}^2\times n_{\max}^2$, where $n_{\max} = \max_{1\leq i\leq d} n_i$. The total complexity of Algorithm~\ref{alg: whole_ring} is $O(d\log d\cdot n_{\max}^6)$. (Recall that the computational complexity of SVD is $O(m_1 m_2\cdot \min\{m_1,m_2\})$ for a $m_1\times m_2$ matrix \cite{trefethen1997numerical}.) It is also possible to further reduce the computational complexity using randomized SVD (see e.g.,  \cite{halko2011finding}), while we will not go into the details here. 

\subsection{Recovery of the underlying path for tensor train format}
\label{sec:alg_TT}

The idea for recovering the underlying graph of TT format is similar to that of TR format. The difference is that, the rotational invariance of TR format no longer holds for TT format, i.e., $\alpha\in S_d$ defined via $\alpha(j) = j+1,\ j\in\{1,2,\dots,d\}$, appears in \eqref{equiv_class_TR}, but not in \eqref{equiv_class_TT}. Therefore, the ordering of four indices in TT format has 12 different cases, much more complicated than in TR format. As an alternative, we consider the relative positions of three indices, which leads to the simplest nontrivial subproblems and is defined below with reflective invariance.

\begin{definition}
For tensor train format with $\tau\in S_d$ being the underlying permutation and three different indices $i_1,i_2,i_3\in \{1,2,\dots,d\}$, we say that $(i_1,i_2,i_3)$ is of the correct order with respect to $\tau$ if $\tau^{-1}(i_1)<\tau^{-1}(i_2)<\tau^{-1}(i_3)$ or $\tau^{-1}(i_3)<\tau^{-1}(i_2)<\tau^{-1}(i_1)$.
\end{definition}

Similarly, the relative positions of all triples of indices provide all information about the underlying permutation, up to reflection.

\begin{proposition}\label{prop:equiv_class_TT}
Let $\tau,\tau'\in S_d$. $\tau'\in \CPTR^d(\tau)$ if and only if for any $1\leq j_1<j_2<j_3\leq d$, $(\tau'(j_1),\tau'(j_2),\tau'(j_3))$ is of the the correct order with respect to $\tau$.
\end{proposition}

For any three indices $i_1,i_2,i_3$, determining their order can still be done by investigating three matrices obtained by downsampling and reshaping/matricizing, analogous to the TR case: $M_{i_1, (i_2,i_3)}\in\bR^{n_{i_1}\times (n_{i_2}n_{i_3})}$, $M_{i_2,(i_3,i_1)}\in\bR^{n_{i_2}\times (n_{i_3} n_{i_1})}$, and $M_{i_3,(i_1,i_2)}\in\bR^{n_{i_3}\times (n_{i_1} n_{i_2})}$ with entries being
 \begin{align}\label{M123_TT}
 & \begin{aligned}
    M_{i_1,(i_2,i_3)}&(x_{i_1},(x_{i_2},x_{i_3})) \\
    & = \tau(\tu{},
     \tau)(y_1,\dots,y_{i_1'-1},x_{i_1'}, y_{i_1'+1},\dots,y_{i_2'-1}, \\
     & \qquad \qquad x_{i_2'},y_{i_2'+1},\dots,y_{i_3'-1},x_{i_3'},y_{i_3'+1},\dots,y_d), 
 \end{aligned}\\
 & \label{M231_TT}
     M_{i_2,(i_3,i_1)}(x_{i_2},(x_{i_3},x_{i_1})) = M_{i_1,(i_2,i_3)}(x_{i_1},(x_{i_2},x_{i_3})), \\
 \intertext{and} 
 & \label{M312_TT}
     M_{i_3,(i_1,i_2)}(x_{i_3},(x_{i_1},x_{i_2})) = M_{i_1,(i_2,i_3)}(x_{i_1},(x_{i_2},x_{i_3})),
 \end{align}
 for $1\leq x_{i_s}\leq n_{i_s}$, $1\leq s\leq 3$, where $\{i_1',i_2',i_3'\}=\{i_1,i_2,i_3\}$, $1\leq i_1'<i_2'<i_3'\leq d$, and $y_i\in \{1,2,\dots,n_i\}$ is a fixed index for each $i\in\{1,2,\dots,d\}\backslash\{i_1,i_2,i_3\}$. Similar to the TR case, one can obtain some intuition by considering the case with $r_2= r_3 = \dots=r_d = R$ and sufficiently large $\vec{n}$: if $(i_1,i_2,i_3)$ is of the correct order than the ranks of $M_{i_1,(i_2,i_3)}$ and $M_{i_3,(i_1,i_2)}$ are both no larger than $R$ while the rank of $M_{i_2,(i_3,i_1)}$ could be $R^2$. This is shown in Figure~\ref{fig:intuition_TT}.
 \begin{figure}[htb!]
    \centering
    \begin{subfigure}[b]{0.48\textwidth}
    \centering
    \begin{tikzpicture}[
        roundnode/.style={circle, draw = green},
    ]
    
    \draw (-1.8,0) node[roundnode] (i1) {$\tu{i_1}$};
    \draw (1.8,0.75) node[roundnode] (i2) {$\tu{i_2}$};
    \draw (1.8,-0.75) node[roundnode] (i3) {$\tu{i_3}$};
    
    \draw[blue,-] (i1.east) -- (i2.west);
    \draw[blue,-] (i2.south) -- (i3.north);
    \draw[red, dashed] (1.1,-1.5) rectangle (2.5,1.5);
    \draw[-] (i1.west) -- (-3,0);
    \draw[-] (i2.east) -- (3,0.75);
    \draw[-] (i3.east) -- (3,-0.75);
    \end{tikzpicture}
    \caption{$M_{i_1,(i_2,i_3)}$}
    \end{subfigure}
    \hfill
    \begin{subfigure}[b]{0.48\textwidth}
    \centering
    \begin{tikzpicture}[
        roundnode/.style={circle, draw = green},
    ]
    
    \draw (-1.8,-0.75) node[roundnode] (i1) {$\tu{i_1}$};
    \draw (-1.8,0.75) node[roundnode] (i2) {$\tu{i_2}$};
    \draw (1.8,0) node[roundnode] (i3) {$\tu{i_3}$};
    
    \draw[blue,-] (i1.north) -- (i2.south);
    \draw[blue,-] (i2.east) -- (i3.west);
    \draw[red, dashed] (-2.5,-1.5) rectangle (-1.1,1.5);
    \draw[-] (i1.west) -- (-3,-0.75);
    \draw[-] (i2.west) -- (-3,0.75);
    \draw[-] (i3.east) -- (3,0);
    \end{tikzpicture}
    \caption{$M_{i_3,(i_1,i_2)}$}
    \end{subfigure}
    \hfill
    \begin{subfigure}[b]{0.48\textwidth}
    \centering
    \begin{tikzpicture}[
        roundnode/.style={circle, draw = green},
    ]
    
    \draw (-1.8,-0.75) node[roundnode] (i3) {$\tu{i_3}$};
    \draw (-1.8,0.75) node[roundnode] (i1) {$\tu{i_1}$};
    \draw (1.8,0) node[roundnode] (i2) {$\tu{i_2}$};
    
    \draw[blue,-] (i1.east) -- (i2.north west);
    \draw[blue,-] (i2.south west) -- (i3.east);
    \draw[red, dashed] (-2.5,-1.5) rectangle (-1.1,1.5);
    \draw[-] (i1.west) -- (-3,0.75);
    \draw[-] (i3.west) -- (-3,-0.75);
    \draw[-] (i2.east) -- (3,0);
    \end{tikzpicture}
    \caption{$M_{i_2,(i_3,i_1)}$}
    \end{subfigure}
    \caption{Illustration of the intuition of Algorithm~\ref{alg: 3index}: when $(i_1,i_2,i_3)$ is of the correct order, the rank of $M_{i_1,(i_2,i_3)}$ and $M_{i_3,(i_1,i_2)}$ is at most $R$, while the rank of $M_{i_2,(i_3,i_1)}$ is generically at least $R^2$. The reason is that the number of blue lines connecting the red box and the node outside the red box is one in (A) and (B), but two in (C), where red boxes group the column indices in the matricization.} 
    \label{fig:intuition_TT}
\end{figure}
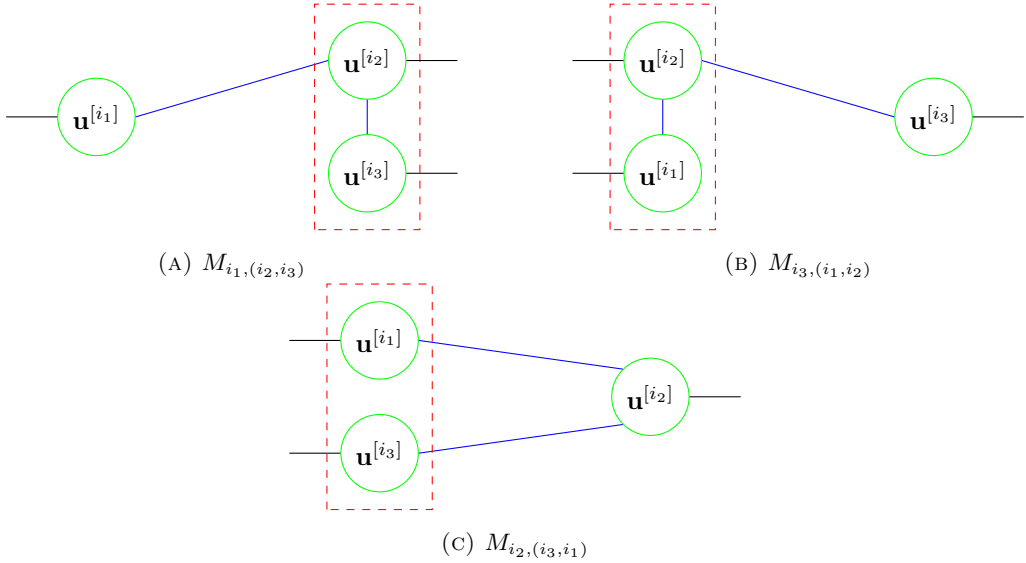
A more precise procedure is stated in Algorithm~\ref{alg: 3index} that is based on $R\in\mathbb{N}_+$ that is an approximation to the bond dimension.

 \begin{algorithm}[htb!]
\caption{Determine the order of three indices for TT format}
 \label{alg: 3index}
 \begin{algorithmic}[1]
 \State \textbf{Input:} Three indices $i_1,i_2,i_3\in\{1,2,\dots,d\}$, $R\in\mathbb{N}_+$ approximating the bond dimension, and an oracle computing entries of $\tT = \varphi(\tu{},\tau)$.
 \State Sample $y_i\in \{1,2,\dots,n_i\}$ for $i\in\{1,2,\dots,d\}\backslash\{i_1,i_2,i_3\}$.
 \State Construct matrices $M_{i_1,(i_2,i_3)}$, $M_{i_2,(i_3,i_1)}$, and $M_{i_3,(i_1,i_2)}$ via \eqref{M123_TT}, \eqref{M231_TT}, and \eqref{M312_TT}.
 \If {$\sigma_{R^2}(M_{i_1,(i_2,i_3)})>\max\left\{\sigma_{R^2}(M_{i_2,(i_3,i_1)}),\sigma_{R^2}(M_{i_3,(i_1,i_2)})\right\}$}
     \State Return $(i_2,i_1,i_3)$.
 \EndIf
 \If {$\sigma_{R^2}(M_{i_2,(i_3,i_1)})>\max\left\{\sigma_{R^2}(M_{i_3,(i_1,i_2)}),\sigma_{R^2}(M_{i_1,(i_2,i_3)})\right\}$}
     \State Return $(i_1,i_2,i_3)$.
 \EndIf
 \If {$\sigma_{R^2}(M_{i_3,(i_1,i_2)})>\max\left\{\sigma_{R^2}(M_{i_1,(i_2,i_3)}),\sigma_{R^2}(M_{i_2,(i_3,i_1)})\right\}$}
     \State Return $(i_2,i_3,i_1)$.
 \EndIf
 \end{algorithmic}
 \end{algorithm}
 
 Similarly, Algorithm~\ref{alg: 3index} can also be implemented without knowing $R$ in the noiseless case. The idea of building a final algorithm based on Algorithm~\ref{alg: 3index} is also similar to that of TR format and we would need the following definition.
 
 \begin{definition}\label{def:consistent_TT}
 Let $\tau\in S_d$ and let $(i_1,i_2,\dots,i_t)$ be a permutation of $t$ different indices in $\{1,2,\dots,d\}$, $3\leq t\leq d$. Then $(i_1,i_2,\dots,i_t)$ is said to be consistent with $\tau$ if $(i_{j_1}, i_{j_2}, i_{j_3})$ is of the correct order with respect to $\tau$ for any $1\leq j_1<j_2<j_3\leq t$.
 \end{definition}
 
 Similar to the TR case, we may recover the whole path by inserting indices one by one. 
 Suppose that for $3\leq t\leq d-1$, $(1,2,\dots,t)$ is permuted to $(i_1,i_2,\dots,i_t)$ that is consistent with $\tau$, and one aims to insert $t+1$ and keeps the consistency. For any $1\leq j_1<j_2\leq t$, the output of Algorithm~\ref{alg: 3index} with the three input indices being $t+1$, $i_{j_1}$, and $i_{j_2}$ indicates that whether $t+1$ should be placed before $i_{j_1}$, between $i_{j_1}$ and $i_{j_2}$, or after $i_{j_2}$. By using a similar idea as in the binary search, one could identify the correct position to insert $t+1$ by calling Algorithm~\ref{alg: 3index} for $O(\log n)$ times. Algorithm~\ref{alg: whole_train} displays the whole procedure.
 
 \begin{algorithm}[htb!]
\caption{Recover the underlying graph for TT format}
 \label{alg: whole_train}
 \begin{algorithmic}[1]
 \State \textbf{Input:} $R\in\mathbb{N}_+$ approximating the bond dimension and a oracle computing entries of $\tT = \varphi(\tu{},\tau)$.
 \State Let $(i_1,i_2,i_3)$ be the output of Algorithm~\ref{alg: 3index} with the input three indices being $1,2,3$.
 \For {$t = 3 : d-1$}
    \State $j_{\min}=0$, $j_{\max} = t+1$. \Comment{Use the idea of binary search to find a location in $(i_1,i_2,\dots,i_t)$ such that inserting $t+1$ maintains the consistency with $\tau$.}
    \While {$j_{\max}-j_{\min}\geq 2$}
        \If {$j_{\max}-j_{\min}\geq 3$}
            \State $j_1 = j_{\min} + \lfloor\frac{1}{3}\cdot (j_{\max}-j_{\min})\rfloor$, $j_2 = j_{\min} + \lfloor\frac{2}{3}\cdot (j_{\max}-j_{\min})\rfloor$.
        \ElsIf { $j_{\min} \geq 1$}
            \State $j_1 = j_{\min}$, $j_2 = j_{\min} + 1$.
        \Else 
            \State $j_1 = 1$, $j_2 = 2$.
        \EndIf
        \State Run Algorithm~\ref{alg: 3index} with input being $t+1$, $i_{j_1}$, and $i_{j_2}$.
        \If {the output of Algorithm~\ref{alg: 3index} is $(t+1, i_{j_1}, i_{j_2})$ or $(i_{j_2}, i_{j_1},t+1 )$}
            \State $j_{\max} = j_1$.
        \ElsIf {the output of Algorithm~\ref{alg: 3index} is $(i_{j_1},t+1,  i_{j_2})$ or $(i_{j_2},t+1 , i_{j_1})$}
            \State $j_{\min} = j_1$, $j_{\max} = j_2$.
        \Else
            \State $j_{\min} = j_2$.
        \EndIf
    \EndWhile
    \State $(i_1,i_2,\dots,i_{t+1}) = (i_1,i_2,\dots,i_{j_{\min}}, t+1, i_{j_{\max}}, i_{j_{\max}+1},\dots, i_t)$.
\EndFor
\State Return $\tau'\in S_d$ with $\tau'(j) = i_j$, $1\leq j\leq d$.
\end{algorithmic}
\end{algorithm}
 
Similar to Algorithm~\ref{alg: whole_ring}, Algorithm~\ref{alg: whole_train} also collects and combines partial information, and there are other approaches such as majority vote. In addition, the complexity of Algorithm~\ref{alg: whole_train} is also similar to that of Algorithm~\ref{alg: whole_ring}, which is polynomial in $d$. More specifically, implementing Algorithm~\ref{alg: whole_train} requires observing $O(d\log d\cdot n_{\max}^3)$ entries of $\tT = \varphi(\tu{},\tau)$ and computing the $R^2$-th largest singular value of $O(d\log d)$ matrices whose sizes are at most $n_{\max}\times n_{\max}^2$, with $n_{\max} = \max_{1\leq i\leq d} n_i$, leading to the overall complexity $O(d\log d\cdot n_{\max}^4)$.

\section{Analysis of the Algorithms}
\label{sec:theory}

In this section, we establish the theoretical guarantees for the algorithms displayed in Section~\ref{sec:alg}. In Section~\ref{sec:correct} we analyze the noiseless case and establish that Algorithm~\ref{alg: 4index}-\ref{alg: whole_train} recover the correct order unless the data lies in a measure-zero set with respect to the Lebesgue measure (note that $\calU_{\vec{r},\vec{n},\tau}^d$ is a Euclidean space). 
Section~\ref{sec:robust} considers the noisy case and provides high-probability guarantees when the observation is contaminated by Gaussian noise. 

\subsection{Recovery in the noiseless case}
\label{sec:correct}

We show that Algorithm~\ref{alg: 4index} returns the correct order of any four given indices when the observation is noiseless (Theorem~\ref{thm:TR}). This result further implies  Algorithm~\ref{alg: whole_ring} returns the correct complete order of all $d$ nodes (Corollary~\ref{cor:TR}).

Our analysis requires the assumption that the bond dimension $\vec{r}$ is relatively small compared to the physical dimension $\vec{n}$. Although this does not cover all tensors in TT/TR format, there exist some interesting tensors with small bond dimensions, such as the TR representing the Ising spin glass in \cite{khoo2021efficient} and the discrete Laplace with TT format in \cite{oseledets2011tensor}. In addition, we require that the minimal element in the bond dimension $\vec{r}$ is comparable with the maximal one. More precise statements of the assumptions can be found in Assumption~\ref{asp:exist_R_TR} for TR format (and in Assumption~\ref{asp:exist_R_TT} for TT format).

\begin{assumption}\label{asp:exist_R_TR}
The input $R\in\mathbb{N}_+$ in Algorithm~\ref{alg: 4index} and Algorithm~\ref{alg: whole_ring} satisfies
\begin{equation*}
    \min_{1\leq i\leq d}n_i\geq R^2> \max_{1\leq j\leq d} r_j,\quad\text{and}\quad \min_{1\leq j\leq d} r_j\geq R.
\end{equation*}
\end{assumption}

\begin{theorem}\label{thm:TR}
For tensor ring format, suppose that Assumption~\ref{asp:exist_R_TR} holds and one can compute the entries of $\mathbf{T} =\varphi(\tu{},\tau)$ exactly. Then for any $\tau\in S_d$ and any four indices $i_1,i_2,i_3,i_4\in \{1,2,\dots,d\}$, the output of Algorithm~\ref{alg: 4index} is of the correct order with respect to $\tau$ unless $\tu{}$ is in a measure-zero set of $\calU_{\vec{r},\vec{n},\tau}^d$.
 \end{theorem}
 
 \begin{corollary}\label{cor:TR}
 In the same setting as in Theorem~\ref{thm:TR}, for any $\tau\in S_d$, Algorithm~\ref{alg: whole_ring} can return some element in $\CPTR^d(\tau)$ unless $\tu{}$ is in a measure-zero set of $\calU_{\vec{r},\vec{n},\tau}^d$.
 \end{corollary}
 
For proving Theorem~\ref{thm:TR}, we assume $(i_1,i_2,i_3,i_4)$ is of the correct order with respect to $\tau$, $i_s = \tau(j_s)$ for $s=1,2,3,4$, and $1\leq j_1 <j_2 < j_3 < j_4\leq d$ without loss of generality. The following Lemma~\ref{lem:rk_upbd_TR} and Lemma~\ref{lem:rk_lowbd_TR}  establish the upper bounds of $\rk(M_{(i_1,i_2),(i_3,i_4)})$, $\rk(M_{(i_1,i_4),(i_2,i_3)})$ and a lower bound of $\rk(M_{(i_1,i_3),(i_2,i_4)})$, respectively.
 
 \begin{lemma}\label{lem:rk_upbd_TR}
 Suppose that $1\leq j_1 <j_2 < j_3 < j_4\leq d$ and that $i_s = \tau(j_s)$, $s=1,2,3,4$. Then for any $\tu{}\in\calU_{\vec{r},\vec{n},\tau}^d$, it holds that 
 \begin{equation}\label{upper_bd_TR_rk1}
     \rk(M_{(i_1,i_2),(i_3,i_4)})\leq \min_{j_2 + 1\leq j \leq j_3} r_j\cdot\min_{j_4+1\leq j \leq j_1} r_j,
 \end{equation}
and similarly that
 \begin{equation}\label{upper_bd_TR_rk2}
     \rk(M_{(i_1,i_4),(i_2,i_3)})\leq \min_{j_1+1\leq j \leq j_2} r_j\cdot\min_{j_3+1\leq j \leq j_4} r_j,
 \end{equation}
 where the index $j$ is understood up to $\text{mod }d$.
\end{lemma}

\begin{proof} 
We only prove \eqref{upper_bd_TR_rk1}. The proof of \eqref{upper_bd_TR_rk2} follows the same argument and is thus omitted. Set
 \begin{equation*}
     j' = \arg\min_{j_2+1\leq j \leq j_3} r_j, \quad \text{and}\quad j'' = \arg\min_{j_4+1\leq j \leq j_1} r_j.
 \end{equation*}
The main idea of the proof is to express the matricization of $\mathbf{T} =\varphi(\tu{},\tau)$, with two indices grouping $x_{\tau(j'')},x_{\tau(j''+1)},\dots,x_{\tau(j'-1)}$ and $x_{\tau(j')},x_{\tau(j'+1)},\dots,x_{\tau(j''-1)}$, as the product of two matrices through reindexing of the TR format. The sizes of two matrices are $(n_{\tau(j'')}n_{\tau(j''+1)}\cdots n_{\tau(j'-1)})\times (r_{j'} r_{j''})$ and $(r_{j'}r_{j''})\times (n_{\tau(j')}n_{\tau(j'+1)}\cdots n_{\tau(j''-1)})$, respectively. More specifically, consider
\begin{equation*}
    \Tilde{M}\in\bR^{(n_{\tau(j'')}n_{\tau(j''+1)}\cdots n_{\tau(j'-1)})\times (n_{\tau(j')}n_{\tau(j'+1)}\cdots n_{\tau(j''-1)})},
\end{equation*}
defined via
 \begin{align*}
     \Tilde{M}\bigl((x_{\tau(j'')}, x_{\tau(j''+1)},& \dots,x_{\tau(j'-1)}), (x_{\tau(j')},x_{\tau(j'+1)},\dots,x_{\tau(j''-1)})\bigr)\\
     &=\tau(\tu{},\tau)(x_1,x_2,\dots,x_d)\\
     &=\tr(\tu{\tau(1)}(x_{\tau(1)})\tu{\tau(2)}(x_{\tau(2)})\cdots \tu{\tau(d)}(x_{\tau(d)}))\\
     &=\tr \left(\tu{\tau(j'')}(x_{\tau(j'')})\tu{\tau(j''+1)}(x_{\tau(j''+1)})\cdots \tu{\tau(j'-1)}(x_{\tau(j'-1)})\right.\\
     &\qquad\ \left.\cdot \tu{\tau(j')}(x_{\tau(j'}))\tu{\tau(j'+1)}(x_{\tau(j'+1)})\cdots \tu{\tau(j''-1)}(x_{\tau(j''-1)})\right),
 \end{align*}
 for $1\leq x_i\leq n_i$, $1\leq i\leq d$. Note that 
 \begin{equation*}
     \tu{\tau(j'')}(x_{\tau(j'')})\tu{\tau(j''+1)}(x_{\tau(j''+1)})\cdots \tu{\tau(j'-1)}(x_{\tau(j'-1)})\in\bR^{r_{j''}\times r_{j'}},
 \end{equation*}
 and
 \begin{equation*}
     \tu{\tau(j')}(x_{\tau(j'}))\tu{\tau(j'+1)}(x_{\tau(j'+1)})\cdots \tu{\tau(j''-1)}(x_{\tau(j''-1)})\in\bR^{r_{j'}\times r_{j''}}.
 \end{equation*}
 Let us define two matrices
 \begin{equation*}
     N_1\in\bR^{(n_{\tau(j'')}n_{\tau(j''+1)}\cdots n_{\tau(j'-1)})\times (r_{j'} r_{j''})},
 \end{equation*}
 and
 \begin{equation*}
     N_2\in\bR^{(r_{j'}r_{j''})\times (n_{\tau(j')}n_{\tau(j'+1)}\cdots n_{\tau(j''-1)})},
 \end{equation*}
 as
 \begin{multline*}
     N_1\bigl((x_{\tau(j'')}, x_{\tau(j''+1)}, \dots,x_{\tau(j'-1)}), ( k_{j'}, k_{j''})\bigr) \\
      := \left(\tu{\tau(j'')}(x_{\tau(j'')})\tu{\tau(j''+1)}(x_{\tau(j''+1)})\cdots \tu{\tau(j'-1)}(x_{\tau(j'-1)})\right)_{k_{j''}, k_{j'}}
 \end{multline*}
 and
 \begin{multline*}
     N_2\bigl((k_{j'}, k_{j''}), (x_{\tau(j')},x_{\tau(j'+1)},\dots,x_{\tau(j''-1)})\bigr) \\
     := \left(\tu{\tau(j')}(x_{\tau(j'}))\tu{\tau(j'+1)}(x_{\tau(j'+1)})\cdots \tu{\tau(j''-1)}(x_{\tau(j''-1)})\right)_{k_{j'}, k_{j''}},
 \end{multline*}
 where $1\leq x_i\leq n_i$, $1\leq i\leq d$, and $1\leq k_{j'}\leq r_{j'}$, $1\leq k_{j''}\leq r_{j''}$. Therefore, it holds that
 \begin{align*}
     & \Tilde{M}\bigl((x_{\tau(j'')}, x_{\tau(j''+1)}, \dots,x_{\tau(j'-1)}), (x_{\tau(j')},x_{\tau(j'+1)},\dots,x_{\tau(j''-1)})\bigr)\\
     & = \tr \left(\tu{\tau(j'')}(x_{\tau(j'')})\tu{\tau(j''+1)}(x_{\tau(j''+1)})\cdots \tu{\tau(j'-1)}(x_{\tau(j'-1)})\right.\\
     &\qquad\qquad\qquad \left.\cdot \tu{\tau(j')}(x_{\tau(j'}))\tu{\tau(j'+1)}(x_{\tau(j'+1)})\cdots \tu{\tau(j''-1)}(x_{\tau(j''-1)})\right)\\
     &=\sum_{k_{j'}=1}^{r_{j'}} \sum_{k_{j''}=1}^{r_{j''}} \left(\tu{\tau(j'')}(x_{\tau(j'')})\tu{\tau(j''+1)}(x_{\tau(j''+1)})\cdots \tu{\tau(j'-1)}(x_{\tau(j'-1)})\right)_{k_{j''}, k_{j'}} \\
     &\qquad\qquad\qquad\cdot \left(\tu{\tau(j')}(x_{\tau(j'}))\tu{\tau(j'+1)}(x_{\tau(j'+1)})\cdots \tu{\tau(j''-1)}(x_{\tau(j''-1)})\right)_{k_{j'}, k_{j''}} \\
     &= N_1\bigl((x_{\tau(j'')}, x_{\tau(j''+1)}, \dots,x_{\tau(j'-1)}), : \bigr) N_2\bigl(:, (x_{\tau(j')},x_{\tau(j'+1)},\dots,x_{\tau(j''-1)})\bigr).
 \end{align*}
 This implies that $\Tilde{M} = N_1 N_2$ and hence that $\rk(\Tilde{M})\leq r_{j'}r_{j''}$, which implies \eqref{upper_bd_TR_rk1} as $M_{(i_1,i_2),(i_3,i_4)}$ is a submatrix of $\Tilde{M}$. 
 \end{proof}
 
 \begin{lemma}\label{lem:rk_lowbd_TR}
Suppose that $1\leq j_1 <j_2 < j_3 < j_4\leq d$ and that $i_s = \tau(j_s)$, $s=1,2,3,4$. Let $R_1,R_2,R_3,R_4$ satisfy
 \begin{equation*}
     R_s\leq \min_{j_s+1\leq j \leq j_{s+1}} r_j,\quad\text{and}\quad n_{i_s}\geq R_{s-1} R_s,\quad 1\leq s\leq 4,
 \end{equation*}
 where the subscript of $R_\cdot$ and $j_\cdot$ is understood via $\text{mod }4$. Then the following holds unless $\tu{}$ is in a measure-zero set of $\calU_{\vec{r},\vec{n},\tau}^d$:
 \begin{equation}\label{lower_bd_TR_rk}
     \rk(M_{(i_1,i_3),(i_2,i_4)})\geq R_1 R_2 R_3 R_4.
 \end{equation}
 \end{lemma}
 
\begin{remark}
The main technique used in the proof of Lemma~\ref{lem:rk_lowbd_TR} is to establish the equivalence between the points in $\calU_{\vec{r},\vec{n},\tau}^d$ violating \eqref{lower_bd_TR_rk} with common roots of a class of polynomials. Note that the measure of the root set of a non-zero polynomial is zero, one only needs to show that some polynomial is non-zero, for which constructing a point that is not a common root suffices. 
\end{remark}

\begin{proof}[Proof of Lemma~\ref{lem:rk_lowbd_TR}]
One has $ n_{i_1} n_{i_3}\geq \Tilde{R}$ and $n_{i_2} n_{i_4}\geq \Tilde{R}$ with $\Tilde{R} = R_1 R_2 R_3 R_4$. We denote $N=\binom{n_1 n_2}{\Tilde{R}}\cdot\binom{n_3 n_4}{\Tilde{R}}$, which is the number of $R\times R$ minors in a $n_1 n_2\times n_3 n_4$ matrix. Consider
 \begin{equation*}
     \begin{split}
         f_\ell: \calU_{\vec{r},\vec{n},\tau}^d &\rightarrow\qquad\qquad \bR,\\
         \tu{}\ \ \ &\mapsto \det(M_{(i_1,i_3),(i_2,i_4)}^{(\ell)}),
     \end{split}
 \end{equation*}
 where $M_{(i_1,i_3),(i_2,i_4)}^{(\ell)}$ denotes the $\ell$-th $R\times R$ minor of $M_{(i_1,i_3),(i_2,i_4)}$, for $1\leq \ell\leq N$. Observe that $f_\ell$ is a homogeneous polynomial over entries of $\tu{}$. 
 
 For $\tu{}\in\calU_{\vec{r},\vec{n},\tau}^d$, the corresponding $M_{(i_1,i_3),(i_2,i_4)}$ satisfies $\rk(M_{(i_1,i_3),(i_2,i_4)})<R$ if and only if
 \begin{equation*}
     \tu{}\in \mathbb{V}(f_1,f_2,\dots,f_N)\subset\calU_{\vec{r},\vec{n},\tau}^d,
 \end{equation*}
 where $\mathbb{V}(f_1,f_2,\dots,f_N)$ is the zero locus of the ideal generated by $\{f_1,f_2,\dots,f_N\}$, i.e., the set of common roots of $f_1,f_2,\dots,f_N$. To prove that $\mathbb{V}(f_1,f_2,\dots,f_N)$ is of measure zero, it suffices to show that at least one of $f_1,f_2,\dots,f_N$ is a non-zero polynomial. Therefore, one only needs to construct a $\tu{}\in\calU_{\vec{r},\vec{n},\tau}^d$ satisfying $\tu{}\notin \mathbb{V}(f_1,f_2,\dots,f_N)$, i.e., the rank of the matrix $M_{(i_1,i_3),(i_2,i_4)}$ associated with $\tu{}$ is at least $\Tilde{R}$.
 
 Let us then construct such a $\tu{}=\left(\tu{1},\tu{2}, \dots, \tu{d}\right)$. We denote $e_{p,r}$ as the vector in $\bR^r$ with the $p$-th entry being $1$ and other entries being $0$. Set
 \begin{multline*}
     \tu{\tau(j_s)}(x_{\tau(j_s)}) \\
     =\begin{cases} e_{p,r_{j_s}} e_{q,r_{j_s + 1}}^\top,&\text{if}\ x_{\tau(j_s)} = (p - 1) R_s +q, \ 1\leq p\leq R_{s-1},\ 1\leq q \leq R_s,\\
     0_{r_{j_s}\times r_{j_s+1}}, &\text{if}\  R_{s-1} R_s < x_{\tau(j_s)} \leq n_{i_s},\end{cases}
 \end{multline*}
 for $s=1,2,3,4$, and
 \begin{equation*}
     \tu{i}(y_i) = \sum_{p=1}^{R_s} e_{p, r_{\tau^{-1}(i)}} e_{p, r_{\tau^{-1}(i)+1}}^\top,
 \end{equation*}
 for $i\in\{1,2,\dots,d\}\backslash\{i_1,i_2,i_3,i_4\}$, where $s$ is the unique index in $\{1,2,3,4\}$ such that $j_s< \tau^{-1}(i)<j_{s+1}$. Then we can compute for $x_{\tau(j_s)} = (p_s - 1) R_s +q_s$, where $1\leq p_s\leq R_{s-1}$, $1\leq q_s \leq R_s$, and $1\leq s\leq 4$, that
 \begin{align*}
     & M_{(i_1,i_3),(i_2,i_4)}((x_{i_1},x_{i_3}),(x_{i_32},x_{i_4}))\\
     & \quad =   M_{(i_1,i_3),(i_2,i_4)}((x_{\tau(j_1)},x_{\tau(j_3)}),(x_{\tau(j_2)},x_{\tau(j_4)})) \\
     & \quad =  \tr(\prod_{s-1}^4\left(\tu{\tau(j_s)}(x_{\tau(j_s)})\tu{\tau(j_s+1)}(y_{\tau(j_s+1)})\cdots \tu{\tau(j_{s+1}-1)}(y_{\tau(j_{s+1}-1)})\right))\\
     & \quad =  \tr(\prod_{s-1}^4\left(e_{p_s,r_{j_s}} e_{q_s,r_{j_s + 1}}^\top\prod_{j=j_s+1}^{j_{s+1}-1}\sum_{p=1}^{R_s} e_{p, r_j} e_{p, r_{j+1}}^\top\right))\\
     & \quad = \tr \left(e_{p_1, r_{j_1}}e_{q_1,r_{j_1+1}}^\top\left(\sum_{p=1}^{R_1} e_{p, r_{j_1+1}} e_{p, r_{j_2}}^\top\right) e_{p_2, r_{j_2}}e_{q_2,r_{j_2+1}}^\top\left(\sum_{p=1}^{R_2} e_{p, r_{j_2+1}} e_{p, r_{j_3}}^\top\right)\right.\\
     &\qquad\qquad\left.\cdot e_{p_3, r_{j_3}}e_{q_3,r_{j_3+1}}^\top\left(\sum_{p=1}^{R_3} e_{p, r_{j_3+1}} e_{p, r_{j_4}}^\top\right) e_{p_4, r_{j_4}}e_{q_4,r_{j_4+1}}^\top\left(\sum_{p=1}^{R_4} e_{p, r_{j_4+1}} e_{p, r_{j_1}}^\top\right)\right) \\
     & \quad =  \prod_{s=1}^4 e_{q_s,r_{j_s+1}}^\top\left(\sum_{p=1}^{R_s} e_{p, r_{j_s+1}} e_{p, r_{j_{s+1}}}^\top\right) e_{p_{s+1}, r_{j_{s+1}}}\\
     & \quad =  \delta_{q_1, p_2}\delta_{q_2, p_3}\delta_{q_3, p_4}\delta_{q_4, p_1},
 \end{align*}
where $\delta_{\cdot,\cdot}$ is the Kronecker delta function: $\delta_{a,b} = 1$ if $a=b$; $\delta_{a,b}=0$ if $a\neq b$. Note that 
 \begin{equation*}
     M_{(i_1,i_3),(i_2,i_4)}((x_{i_1},x_{i_3}),(x_{i_2},x_{i_4}))= M_{(i_1,i_3),(i_2,i_4)}((x_{\tau(j_1)},x_{\tau(j_3)}),(x_{\tau(j_2)},x_{\tau(j_4)}))=0,
 \end{equation*}
 if there exists $s\in\{1,2,3,4\}$ such that $x_{\tau(j_s)} > R_{s-1} R_s$. Therefore, we obtain that
 \begin{equation*}
     \rk(M_{(i_1,i_3),(i_2,i_4)})=R_1 R_2 R_3 R_4 = \Tilde{R},
 \end{equation*}
 which completes the proof.
\end{proof}

\begin{proof}[Proof of Theorem~\ref{thm:TR}]
In Lemma~\ref{lem:rk_upbd_TR} and Lemma~\ref{lem:rk_lowbd_TR}, we have shown that if Assumption~\ref{asp:exist_R_TR} holds and $(i_1,i_2,i_3,i_4)$ is of the correct order with respect to $\tau$, then almost surely,
 \begin{equation}\label{eq:rk_TR}
     \rk(M_{(i_1,i_3),(i_2,i_4)}) \geq R^4 >\max\left\{\rk(M_{(i_1,i_2),(i_3,i_4)}), \rk(M_{(i_1,i_4),(i_2,i_3)})\right\}.
\end{equation}
This guarantees the correctness of Algorithm~\ref{alg: 4index}.
\end{proof}

The correctness of Algorithms~\ref{alg: 3index} and \ref{alg: whole_train} can be established with similar proof techniques. We state the result in Theorem~\ref{thm:TT} and Corollary~\ref{cor:TT} and defer the proof of Theorem~\ref{thm:TT} to the appendix.

\begin{assumption}\label{asp:exist_R_TT}
 There exists some $R\in\mathbb{N}_+$ such that
 \begin{equation*}
     \min\{n_{\tau(1)},n_{\tau(d)}\}\geq R,\quad \min_{2\leq j\leq d-1}n_{\tau(j)}\geq R^2 >\max_{2\leq j\leq d} r_j\geq \min_{2\leq j\leq d} r_j\geq R.
 \end{equation*}
 \end{assumption}

 \begin{theorem}\label{thm:TT}
 For tensor train format, suppose that Assumption~\ref{asp:exist_R_TT} holds and one can compute the entries of $\mathbf{T} =\varphi(\tu{},\tau)$ exactly. Then for any $\tau\in S_d$ and any three indices $i_1,i_2,i_3\in \{1,2,\dots,d\}$, the output of Algorithm~\ref{alg: 3index} is of the correct order with respect to $\tau$ unless $\tu{}$ is in a measure-zero set of $\calU_{\vec{r},\vec{n},\tau}^d$.
 \end{theorem}
 
 \begin{corollary}\label{cor:TT}
 In the same setting as in Theorem~\ref{thm:TT}, for any $\tau\in S_d$, Algorithm~\ref{alg: whole_train} can return some element in $\CPTT^d(\tau)$ unless $\tu{}$ is in a measure-zero set of $\calU_{\vec{r},\vec{n},\tau}^d$.
 \end{corollary}

\subsection{Robustness against observation error}
\label{sec:robust}

As Algorithms~\ref{alg: 4index}-\ref{alg: whole_train} are based on singular values, which are continuous with respect to matrix entries, they are expected to be robust against observational noise. In this section, we establish rigorous robustness results assuming the noise is Gaussian. Similar to the previous subsection, we will present the proof for the TR format and defer the similar proof for the TT format to the appendix.

Let $\tau\in S_d$ be the underlying permutation and let $R$ be as in Assumption~\ref{asp:exist_R_TR}. It has already been proved in Lemma~\ref{lem:rk_lowbd_TR} that 
\begin{equation}\label{SinValue>0_TR}
     \sigma_{R^4}(M_{(i_1,i_3),(i_2,i_4)}) > 0
\end{equation}
for almost all $\tu{}\in \calU_{\vec{r},\vec{n},\tau}^d$ and all $i_1,i_2,i_3,i_4\in\{1,2,\dots,d\}$ of the correct order with respect to $\tau$. For the noisy case, we need a stronger assumption than \eqref{SinValue>0_TR}, requiring a uniform positive lower bound of the $R^4$-th singular value:
 
\begin{assumption}\label{assump:uniform_SingValue_TR}
Let $\tu{}\in\calU_{\vec{r},\vec{n},\tau}^d$ be the underlying data for the TR format and let Assumption~\ref{asp:exist_R_TR} be satisfied with $R>0$. There exists a positive constant $\sigma>0$ such that for any $i_1,i_2,i_3,i_4\in\{1,2,\dots,d\}$ of correct order with respect to $\tau$ and any $y_i\in \{1,2,\dots,n_i\}$, $i\in\{1,2,\dots,d\}\backslash\{i_1,i_2,i_3,i_4\}$, the matrix $M_{(i_1,i_3),(i_2,i_4)}\in \bR^{(n_{i_1} n_{i_3})\times (n_{i_2} n_{i_4})}$ satisfies 
\begin{equation}\label{sigma_lb_TR}
    \sigma_{R^4}(M_{(i_1,i_3),(i_2,i_4)})\geq \sigma.
\end{equation}
\end{assumption}
 
With the uniform bound assumption above, we can establish a high-probability guarantee of the correctness of Algorithm~\ref{alg: 4index}.

\begin{theorem}\label{thm:robust_TR}
Suppose Assumption~\ref{assump:uniform_SingValue_TR} holds and the noise in each observation of $\mathbf{T} = \varphi(\tu{}, \tau)$ satisfies $\mathcal{N}(0,\sigma_e^2)$ independently. For any $T< \frac{\sigma}{\sigma_e}$ ($\sigma$ is defined in \eqref{sigma_lb_TR}), any four indices $i_1,i_2,i_3,i_4\in\{1,2,\dots,d\}$, and any $y_i\in \{1,2,\dots,n_i\}$, $i\in\{1,2,\dots,d\}\backslash\{i_1,i_2,i_3,i_4\}$, the probability that Algorithm~\ref{alg: 4index} returns an incorrect ordering of $i_1,i_2,i_3,i_4$ is no more than
\begin{equation*}
     6\cdot \exp\left(-\frac{1}{8}\max\{T-4n_{\max}, 0\}^2\right).
\end{equation*}
 \end{theorem}
 
 \begin{corollary}\label{cor:robust_TR}
 In the same setting as in Theorem~\ref{thm:robust_TR}, the probability that Algorithm~\ref{alg: whole_ring} returns an incorrect ordering is smaller than or equal to
\begin{equation*}
     C d\log d\cdot \exp\left(-\frac{1}{8}\max\{T-4n_{\max}, 0\}^2\right),
 \end{equation*}
 where $C$ is an absolute constant.
 \end{corollary}

It is evident from Theorem~\ref{thm:robust_TR} and Corollary~\ref{cor:robust_TR} that the level of robustness against the noise is largely determined by the singular value threshold $\sigma$ in \eqref{sigma_lb_TR}. In some datasets, $\sigma$ may be small, a scenario to be demonstrated through numerical examples later.

For proving Theorem~\ref{thm:robust_TR}, we first state some preliminaries that bound the singular values of a matrix with noisy entries:
\begin{theorem}[Largest singular value of matrices with i.i.d. Gaussian entries, \cites{rudelson2010non,davidson2001local}] \label{thm:large_sing_value_gauss}
Suppose that $A\in\bR^{m_1\times m_2}$ and the entries of $A$ are i.i.d. $\mathcal{N}(0,1)$. Then the largest singular value of $A$ satisfies
\begin{equation*}
    \mathbb{P}(\sigma_{\max}(A)>\sqrt{m_1}+\sqrt{m_2}+t)\leq 2 e^{-t^2/2},\quad \forall~t>0.
\end{equation*}
\end{theorem}
 
\begin{theorem}[Weyl's inequality]\label{thm:weyl_ineq}
 For any $A,E\in \bR^{m_1\times  m_2}$, it holds that
 \begin{equation*}
     |\sigma_k (A+E)-\sigma_k(A)|\leq \sigma_{\max}(E),\quad\forall~1\leq k\leq \min\{m_1,m_2\}.
 \end{equation*}
 \end{theorem}

\begin{proof}[Proof of Theorem~\ref{thm:robust_TR}]
We first show the robustness of Algorithm~\ref{alg: 4index}. Without loss of generality, let us assume that $(i_1,i_2,i_3,i_4)$ is of the correct order with respect to $\tau$. Let $E$ be the noise matrix in observing $M_{(i_1,i_2),(i_3,i_4)}$, $M_{(i_1,i_3),(i_2,i_4)}$, or $M_{(i_1,i_4),(i_2,i_3)}$. Then by Theorem~\ref{thm:large_sing_value_gauss}, it holds that
\begin{equation*}
    \mathbb{P}\left(\sigma_{\max}(E)\geq \frac{\sigma}{2}\right)\leq \mathbb{P}\left(\sigma_{\max}(E) > (2 n_{\max} + t) \sigma_e\right)\leq 2 e^{-t^2/2},
\end{equation*}
where $t = \frac{T}{2}-2 n_{\max}$. Then using Assumption~\ref{assump:uniform_SingValue_TR} and Theorem~\ref{thm:weyl_ineq} (Weyl's inequality), with probability at least
\begin{equation*}
    1- 6\cdot \exp\left(-\frac{t^2}{2}\right) = 1 - 6\cdot \exp\left(-\frac{1}{8}(T-4n_{\max})^2\right),
\end{equation*}
we have $\sigma_{R^4}(M_{(i_1,i_2),(i_3,i_4)})<\frac{\sigma}{2}$, $\sigma_{R^4}(M_{(i_1,i_3),(i_2,i_4)})>\frac{\sigma}{2}$, and $\sigma_{R^4}(M_{(i_1,i_4),(i_2,i_3)})<\frac{\sigma}{2}$, which leads to the correctness of the output of Algorithm~\ref{alg: 4index}. 
\end{proof}

\begin{proof}[Proof of Corollary~\ref{cor:robust_TR}]
Corollary~\ref{cor:robust_TR} follows immediately from Theorem~\ref{thm:robust_TR} and the fact that Algorithm~\ref{alg: whole_ring} calls Algorithm~\ref{alg: 4index} for $O(d\log d)$ times.
\end{proof}

The robustness result of Algorithm~\ref{alg: 3index} and Algorithm~\ref{alg: whole_train} is similar, which is stated as follows and proved in the appendix.

 \begin{assumption}\label{assump:uniform_SingValue_TT}
Let $\tu{}\in\calU_{\vec{r},\vec{n},\tau}^d$ be the underlying data for the TT format and let Assumption~\ref{asp:exist_R_TT} be satisfied with $R>0$. There exists a positive constant $\sigma>0$ such that for any $i_1,i_2,i_3\in\{1,2,\dots,d\}$ of correct order with respect to $\tau$ and any $y_i\in \{1,2,\dots,n_i\}$, $i\in\{1,2,\dots,d\}\backslash\{i_1,i_2,i_3\}$, the matrix $M_{i_2,(i_1,i_3)}\in \bR^{n_{i_2}\times (n_{i_1} n_{i_3})}$ satisfies 
 \begin{equation}\label{sigma_lb_TT}
     \sigma_{R^2}(M_{i_2,(i_1,i_3)})\geq \sigma.
 \end{equation}
 \end{assumption}

 \begin{theorem}\label{thm:robust_TT}
 Suppose Assumption~\ref{assump:uniform_SingValue_TT} holds and the noise in each observation of $\mathbf{T} = \varphi(\tu{}, \tau)$ satisfies $\mathcal{N}(0,\sigma_e^2)$ independently. For any $T< \frac{\sigma}{\sigma_e}$ ($\sigma$ is defined in \eqref{sigma_lb_TT}), any three indices $i_1,i_2,i_3\in\{1,2,\dots,d\}$, and any $y_i\in \{1,2,\dots,n_i\}$, $i\in\{1,2,\dots,d\}\backslash\{i_1,i_2,i_3\}$, the probability that Algorithm~\ref{alg: 3index} returns an incorrect ordering of $i_1,i_2,i_3$ is at most
 \begin{equation*}
     6\cdot \exp\left(-\frac{1}{8}\max\{T-2n_{\max}-2\sqrt{n_{\max}},0\}^2\right),
 \end{equation*}
 as long as the noise in each observation of $\mathbf{T} = \varphi(\tu{}, \tau)$ distributes as $\mathcal{N}(0,\sigma_e^2)$ independently with $\sigma_e < \frac{\sigma}{T}$, where $\sigma$ is the lower bound in \eqref{sigma_lb_TT}.
 \end{theorem}
 
 \begin{corollary}\label{cor:robust_TT}
 In the same setting as in Theorem~\ref{thm:robust_TT}, the probability that Algorithm~\ref{alg: whole_train} returns an incorrect ordering is smaller than or equal to
\begin{equation*}
     C d\log d\cdot \exp\left(-\frac{1}{8}\max\{T-2n_{\max}-2\sqrt{n_{\max}}, 0\}^2\right),
 \end{equation*}
 where $C$ is an absolute constant.
 \end{corollary}

\section{Numerical Experiments}
\label{sec:numerics}

We present some numerical results in this section\textcolor{blue}. Section~\ref{sec:implement} is for implementing the proposed tensor order recovery algorithms, which shows the correctness and efficiency of the proposed approach. In Section~\ref{sec:compare_ballani}, we compare our work with \cite{ballani2014tree} that heuristically identifies the structure of tree tensor networks. We evaluate our TR order recovery algorithm on a more practical model, the Potts model, in Section~\ref{sec:potts}.

\subsection{Implementation of the proposed algorithms for TR/TT format}
\label{sec:implement}
Throughout this subsection, we set $d=8$ and $\vec{n}=(4,4,\dots,4)$. $\vec{r}$ is chosen as $ (3,3,\dots,3)$ for TR format and $(1,3,3,\dots,3)$ for TT format. One can verify that Assumption~\ref{asp:exist_R_TR} and Assumption~\ref{asp:exist_R_TT} are satisfied for $R=2$. For a fixed noise level $\sigma_e$, we repeat the following procedure for $10000$ times:
\begin{itemize}
    \item Sample an underlying permutation from the uniform distribution $\tau\sim\mathcal{U}(S_d)$.
    \item Sample an element $\tu{}\in\calU_{\vec{r},\vec{n},\tau}^d$ with each entry of $\tu{}$ being i.i.d. $\mathcal{N}(0,1)$.
    \item Run Algorithm~\ref{alg: whole_ring}, Algorithm~\ref{alg: whole_train}, or their majority-vote versions with $5$ voters, where each entry of $\mathbf{T} = \varphi(\tu{},\tau)$ is observed with noise from $\mathcal{N}(0,\sigma_e^2)$. Denote the output by $\tau'$.
\end{itemize}
We count the number of trials with $\tau'\in\CPTR^d(\tau)$ or $\tau'\in\CPTT^d(\tau)$ and divide it by $10000$ to obtain an estimate of the probability that Algorithm~\ref{alg: whole_ring}, Algorithm~\ref{alg: whole_train}, or their majority-vote versions, obtain the correct order/underlying graph. We then repeat the experiments with a modified method for generating the data, while keeping all other settings unchanged. Specifically, the entries of $\tu{}$ are independent, with $\tu{i}(k_i,x_{\tau(i)},k_{i+1})\sim \mathcal{N}(0,0.1^2)$ if $k_i=3$ or $k_{i+1}=3$, and $\tu{i}(k_i,x_{\tau(i)},k_{i+1})\sim \mathcal{N}(0,1)$ otherwise. This dataset is nearly rank-deficient, as it can be seen as a perturbation from a TR format with $\vec{r}=(2,2,\dots,2)$ or a TT format $\vec{r}=(1,2,\dots,2)$, in contrast to the previous dataset, which is viewed full-rank.

The results are shown in Figure~\ref{fig:numericsTR} for TR format and in Figure~\ref{fig:numericsTT} for TT format, respectively. One can see that the probability of correct recovery is $1$ when $\sigma_e=0$ and is high when $\sigma_e$ is small, which fits the theory established in Section~\ref{sec:theory}. It is also evident that the robustness on the full-rank data is significantly better than that on the nearly rank-deficient data. As discussed in Section~\ref{sec:robust}, this is because the singular value threshold is small for the nearly rank-deficient data. In addition, the majority vote can significantly improve the robustness of our approach even if the number of voters is small. 

\begin{figure}[htb!]
    \centering
    \includegraphics[width=0.75\textwidth]{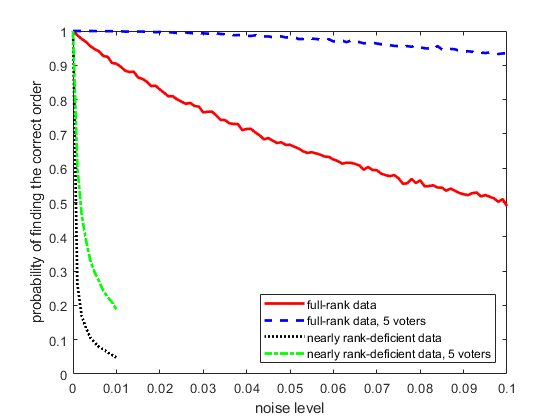}
    \caption{Numerical results for tensor ring format. The probability of correct recovery is $1$ when there is no noise. The robustness is much better on full-rank data compared to nearly rank-deficient data. Using majority vote with $5$ voters significantly improves the robustness of the algorithm.} 
    \label{fig:numericsTR}
\end{figure}

\begin{figure}[htb!]
    \centering
    \includegraphics[width=0.75\textwidth]{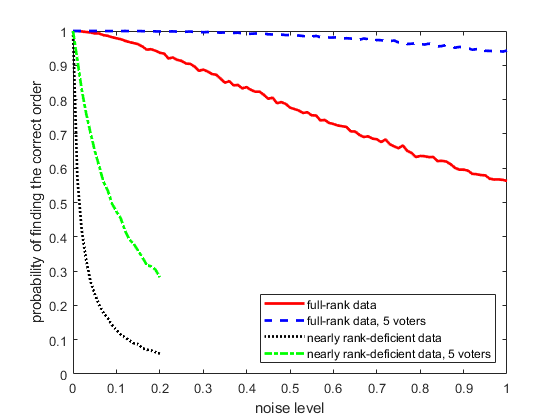}
    \caption{Numerical results for tensor train format. The probability of correct recovery is $1$ when there is no noise.
    The proposed method is more robust on full-rank data than the nearly rank-deficient data. Using majority vote with $5$ voters significantly improves the robustness of the algorithm.} 
    \label{fig:numericsTT}
\end{figure}

Another observation is that the algorithms for TT format tensor network recovery is more robust again the observation noise than those for TR format. More specifically, the curves of probability in Figure~\ref{fig:numericsTR} and Figure~\ref{fig:numericsTT} are similar although the maximal noise levels are significantly different: $0.1$ for TR format and $1$ for TT format. 
This is because when recovering the partial information, i.e., the order of $3$ or $4$ indices,  Algorithm~\ref{alg: 4index} samples a matrix with size of $n_{\max}^2\times n_{\max}^2$, which is much larger than the matrix sampled in Algorithm~\ref{alg: 3index} whose size is at most $n_{\max}\times n_{\max}^2$. According to Theorem~\ref{thm:large_sing_value_gauss} and Theorem~\ref{thm:weyl_ineq}, for a fixed noise level, the perturbation of singular values of a matrix has a larger upper bound when the size of the matrix is larger. The perturbation of singular subspaces follows similar rules (see e.g., \cite{cai2018rate}*{Theorem 3}). 

\subsection{Comparison with previous work for TT format}
\label{sec:compare_ballani}
We compare our approach with the one proposed in \cite{ballani2014tree} in this subsection. The method introduced in \cite{ballani2014tree} aims to recover a general tree structure of a tensor network. The method starts by initializing with a discrete partition $\{\{1\},\{2\},\dots,\{d\}\}$ and subsequently clusters the partition based on some minimal-rank condition. Note that the TR format does not admit a tree structure and hence we only compare our method with the one in \cite{ballani2014tree} for the TT format. An additional insight is that the graph associated with the TT structure is a very special tree, which means it is unnecessary to implement the algorithm in \cite{ballani2014tree} for the most general setting. Alternatively, we first select an endpoint and the TT chain; subsequently, we extend the selected subchain by adding new nodes, guided by the minimal-rank condition. Note that the subchain extension is in the same spirit as clustering the existing subchain with a new node. The numerical comparison between Algorithm~\ref{alg: whole_train} (with five voters) and that of \cite{ballani2014tree} is shown in Figure~\ref{fig:numerics_compare} for $d=6$, $\vec{n}=(4,4,\dots,4)$, $\vec{r}=(1,3,3,\dots,3)$, with other settings being the same as the full-rank experiments in Section~\ref{sec:implement}. It can be seen that both methods by us and by \cite{ballani2014tree} can recover the correct order in the noiseless case, but our method performs more robustly against the observation error. It is worth noting that we implement the method in \cite{ballani2014tree} while estimating ranks based on observation from all entries of $\mathbf{T}$. This process entails much higher complexity compared to Algorithm~\ref{alg: whole_train}, which only makes queries to down-sampled tensors. These distinctions show our proposed approach is more robust and computationally efficient compared to the approach in \cite{ballani2014tree}.

\begin{figure}[htb!]
    \centering
    \includegraphics[width=0.75\textwidth]{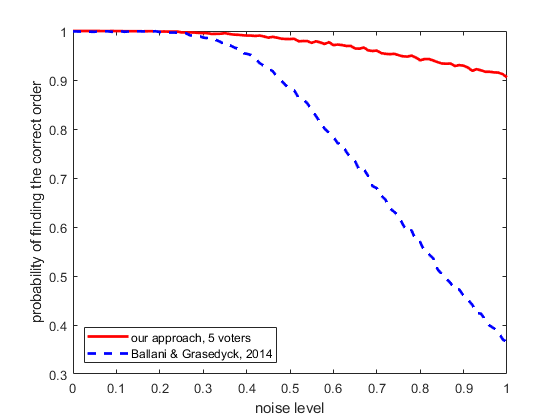}
    \caption{Numerical comparison between Algorithm~\ref{alg: whole_train} (with five voters) and the method proposed in \cite{ballani2014tree} for tensor train format. 
    We observe all entries of $\mathbf{T}$ for the rank estimation when implementing the method in \cite{ballani2014tree} and only observe down-sampled tensors when implementing our approach. Even if we allow the method of \cite{ballani2014tree} to have larger query complexity, our work still supersedes \cite{ballani2014tree} in the sense of having significantly larger probability of finding the correct order.}
    \label{fig:numerics_compare}
\end{figure}

\subsection{Potts model}
\label{sec:potts}

We demonstrate the efficacy of our algorithm for identifying the permutation for a Potts model \cites{jha2022tensor,wu1982potts} when the underlying geometry is unknown. Potts model is a generalization of the Ising model used in statistical physics: Consider a many-body system with $d$ sites/spins and $r$ different spin values and only adjacent sites on a ring (with an unknown permutation $\tau\in S_d$) admit interactions. Specifically, the Hamiltonian is given by
\begin{equation*}
    H_\tau(k_1,k_2,\dots,k_d) = \sum_{i=1}^d J_{\tau(i)}(k_{\tau(i)},k_{\tau(i+1)}),
\end{equation*}
where $k_{\tau(i)}$ is the spin value at the $i$-th site on the ring and $J_{\tau(i)}\in\bR^{r\times r}$ is a symmetric matrix representing the interaction pattern of the $i$-th and $(i+1)$-st sites on the ring. This Hamiltonian leads to the following free energy
\begin{equation}\label{eq:free_energy}
    \begin{split}
        f_\tau(J_1,J_2,\dots,J_d) & = -\frac{1}{\beta} \log \left(\sum_{k_1,k_2,\dots,k_d=1}^r e^{-\beta H_\tau(k_1,k_2,\dots,k_d)}\right) \\
        & = -\frac{1}{\beta} \log \left(\tr\left(\prod_{i=1}^d e^{-\beta J_{\tau(i)}}\right)\right),
    \end{split}
\end{equation}
where $\beta$ is the inverse temperature and $e^{A}$ denotes the element-wise exponential of a matrix $A$.

Our experiment is a generalization of the Ising model experiment in \cite{khoo2021efficient}. We set $d=6$, $r=3$, $\beta = 10$, and let $J_1,J_2,\dots,J_d$ take values among $n=5$ pre-determined $r\times r$ symmetric matrices with diagonal and off-diagonal entries sampled i.i.d. from $\mathcal{N}(0,1)$. The free energy $f_\tau$ can thus be viewed as a $d$-th order tensor with physical dimension being $\vec{n}=(5,\dots,5)$ and we assume that the entries of $f_\tau$ are observed with i.i.d. noise $\mathcal{N}(0,\sigma_e^2)$ with some noise level $\sigma_e$. We then apply Algorithm~\ref{alg: 4index} and Algorithm~\ref{alg: whole_ring} with $R=2$ to recover the underlying permutation $\tau$ and the results of 10000 independent trials are shown in Figure~\ref{fig:Potts}. 

\begin{figure}[htb!]
    \centering
    \includegraphics[width=0.75\textwidth]{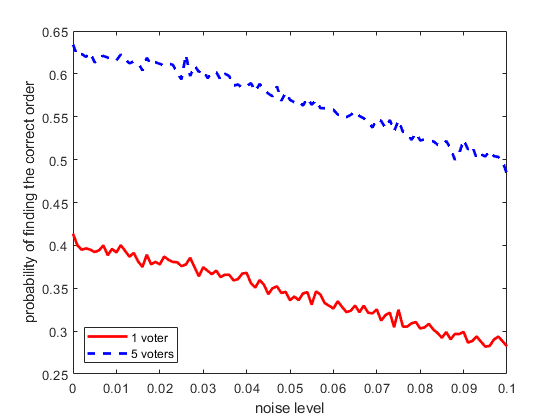}
    \caption{Numerical results for the Potts model ($3$ spin values) with TR format. The probability of correct recovery is relatively high even though the free energy $f_\tau$ is  not precisely of TR format with low bond dimension.}
    \label{fig:Potts}
\end{figure}

It can be seen from Figure~\ref{fig:Potts} that our algorithms can correctly recover the underlying permutation for the Potts model with reasonably high probability. Note that the free energy $f_\tau$ in \eqref{eq:free_energy} is not precisely in the TR format due to the logarithm, which means that Assumption~\ref{asp:exist_R_TR} does not hold. Consequently, the probability of finding the correct order is not $1$ in the noiseless case. Nevertheless, the success probability is still quite high, which further demonstrates the effectiveness of our algorithms.

\section{Conclusion and Discussions}
\label{sec:conclude}

We propose new algorithms for recovering the underlying graphs of the tensor ring and tensor train format, by querying entries of the tensor. The proposed algorithms obtain the orders of the selected $3$ or $4$ indices by constructing matricizations using down-sampling and comparing their ranks. The algorithm follows from the observation that the resulting matricization has a smaller rank if the grouping of indices is compatible with the underlying order of indices. These methods are justified by theory -- almost sure correctness in the noiseless case and high-probability bounds for cases with observation error. In the numerical experiments, we verify these theoretical results and observe that a proper strategy of combining partial information can significantly improve the robustness against the observation error.

We also acknowledge some limitations of this work.  First, specific assumptions regarding the bond dimension and the physical dimension are necessary to ensure the provable efficiency of the proposed algorithm and the uniqueness of the underlying permutation. Second, the level of robustness depends on the singular value threshold, which is generally unclear in a quantitative sense. Third, our analysis does not cover incomplete observations, which could be a promising area for future research given existing methods from tensor/matrix completion and SVD for incomplete observations. Another important future direction is the recovery of underlying graphs for general tensor networks. All these topics could be interesting avenues for future research.

\bibliographystyle{amsxport}
\bibliography{references}

\appendix

\section{Uniqueness of Permutations}
\label{sec:unique_perm}

This section establishes the uniqueness results for the underlying permutation, in the sense that two permutations that can represent the same tensor must be in the same equivalence class defined in \eqref{equiv_class_TR} or \eqref{equiv_class_TT}. The results require the same assumptions as in Section~\ref{sec:correct} and are stated in Theorem~\ref{thm:unique_perm_TR} and Theorem~\ref{thm:unique_perm_TT} for TR and TT format, respectively.

\begin{theorem}\label{thm:unique_perm_TR}
    Let $\tau,\tau'\in S_d$ and let $\tu{} \in \calU_{\vec{r},\vec{n},\tau}^d, \tu{}' \in \calU_{\vec{r}',\vec{n},\tau'}^d$ satisfy $\varphi(\tu{},\tau) = \varphi(\tu{}',\tau')$. Suppose that $(\vec{r},\vec{n})$ and $(\vec{r}',\vec{n})$ both satisfy Assumption~\ref{asp:exist_R_TR} with the same $R\in\mathbb{N}_+$. Then we have $\tau'\in \CPTR^d(\tau)$ unless $\tu{}$ or $\tu{}'$ is located in some measure-zero subset of $\calU_{\vec{r},\vec{n},\tau}^d$ or $\calU_{\vec{r}',\vec{n},\tau'}^d$.
\end{theorem}

\begin{proof}
    Suppose that $\tau'\notin \CPTR^d(\tau)$. Then there must exist four distinct indices $i_1,i_2,i_3,i_4\in\{1,2,\dots,d\}$ such that $(i_1,i_2,i_3,i_4)$ is of the correct order with respect to $\tau$ but is not of the correct order with respect to $\tau'$. By Lemma~\ref{lem:rk_upbd_TR} and Lemma~\ref{lem:rk_lowbd_TR}, it holds almost surely that
    \begin{equation}\label{eq:MM'_TR}
        \rk\big(M_{(i_1,i_2),(i_3,i_4)}'\big) \geq R^4 > \rk\big(M_{(i_1,i_2),(i_3,i_4)}\big),
    \end{equation}
    where $M_{(i_1,i_2),(i_3,i_4)}$ and $M_{(i_1,i_2),(i_3,i_4)}'$ are the matricization of $\varphi(\tu{},\tau)$ and $\varphi(\tu{}',\tau')$ defined in \eqref{M1234_TR}. However, it follows from $\varphi(\tu{},\tau) = \varphi(\tu{}',\tau')$ that $M_{(i_1,i_2),(i_3,i_4)}=M_{(i_1,i_2),(i_3,i_4)}'$, which contradicts \eqref{eq:MM'_TR}.
\end{proof}

\begin{theorem}\label{thm:unique_perm_TT}
    Let $\tau,\tau'\in S_d$ and let $\tu{} \in \calU_{\vec{r},\vec{n},\tau}^d, \tu{}' \in \calU_{\vec{r}',\vec{n},\tau'}^d$ with $r_1=r_1'=1$ satisfy $\varphi(\tu{},\tau) = \varphi(\tu{}',\tau')$. Suppose that $(\vec{r},\vec{n})$ and $(\vec{r}',\vec{n})$ both satisfy Assumption~\ref{asp:exist_R_TT} with the same $R\in\mathbb{N}_+$. Then we have $\tau'\in \CPTT^d(\tau)$ unless $\tu{}$ or $\tu{}'$ is located in some measure-zero subset of $\calU_{\vec{r},\vec{n},\tau}^d$ or $\calU_{\vec{r}',\vec{n},\tau'}^d$.
\end{theorem}

\begin{proof}
    Suppose that $\tau'\notin \CPTT^d(\tau)$. Then there must exist three distinct indices $i_1,i_2,i_3\in\{1,2,\dots,d\}$ such that $(i_1,i_2,i_3)$ is of the correct order with respect to $\tau$ but is not of the correct order with respect to $\tau'$. By Lemma~\ref{lem:rk_upbd_TT} and Lemma~\ref{lem:rk_lowbd_TT}, it holds almost surely that
    \begin{equation}\label{eq:MM'_TT}
        \rk\big(M_{i_1,(i_2,i_3)}'\big) \geq R^2 > \rk\big(M_{i_1,(i_2,i_3)}\big),
    \end{equation}
    where $M_{i_1,(i_2,i_3)}$ and $M_{i_1,(i_2,i_3)}'$ are the matricization of $\varphi(\tu{},\tau)$ and $\varphi(\tu{}',\tau')$ defined in \eqref{M123_TT}. However, it follows from $\varphi(\tu{},\tau) = \varphi(\tu{}',\tau')$ that $M_{i_1,(i_2,i_3)}=M_{i_1,(i_2,i_3)}'$, which contradicts \eqref{eq:MM'_TT}.
\end{proof}

\section{Deferred Proofs}
We collect all deferred proofs in this section.

\subsection{Proof of Theorem~\ref{thm:TT}}
The proof follows similar ideas and techniques as in the proof of Theorem~\ref{thm:TR}. We also use two lemmas to establish the upper bound of the rank of $M_{i_1,(i_2,i_3)}$ and $M_{i_3,(i_1,i_2)}$, as well as the lower bound of the rank of $M_{i_2,(i_1,i_3)}$, given that $(i_1,i_2,i_3)$ is of the correct order with respect to the underlying permutation $\tau$.

 \begin{lemma}\label{lem:rk_upbd_TT}
 Suppose that $i_s = \tau(j_s)$ for $s=1,2,3$, where $1\leq j_1<j_2<j_3\leq d$. Then for any $\tu{}\in\calU_{\vec{r},\vec{n},\tau}^d$, it holds that 
 \begin{equation}\label{rank_M1_upperbd_TT}
     \rk(M_{i_1,(i_2,i_3)})\leq \min_{j_1+1\leq j \leq j_2} r_j,
 \end{equation}
 and that
 \begin{equation}\label{rank_M3_upperbd_TT}
     \rk(M_{i_3,(i_1,i_2)})\leq \min_{j_2+1\leq j \leq j_3} r_j.
 \end{equation}
 \end{lemma}
 
 \begin{proof}
 We only prove \eqref{rank_M1_upperbd_TT} as \eqref{rank_M3_upperbd_TT} will hold by the same reasoning. Consider any $j\in\{j_1+1,j_1+2,\dots,j_2\}$ and the matrix $\Tilde{M}\in\bR^{(n_{\tau(1)}n_{\tau(2)}\cdots n_{\tau(j-1)})\times (n_{\tau(j)}n_{\tau(j+1)}\cdots n_{\tau(d)})}$ defined via
 \begin{align*}
     &\Tilde{M}\left((x_{\tau(1)}, x_{\tau(2)},\dots,x_{\tau(j-1)}),(x_{\tau(j)},x_{\tau(j+1)},\dots,x_{\tau(d)})\right)\\
     =&\tau(\tu{},\tau)(x_1,x_2,\dots,x_d)\\
     =&\tu{\tau(1)}(x_{\tau(1)})\tu{\tau(2)}(x_{\tau(2)})\cdots \tu{\tau(j-1)}(x_{\tau(j-1)})\\
     & \cdot \tu{\tau(j)}(x_{\tau(j}))\tu{\tau(j+1)}(x_{\tau(j+1)})\cdots \tu{\tau(d)}(x_{\tau(d)}),
 \end{align*}
  for $1\leq x_i\leq n_i$, $1\leq i\leq d$. Note that 
 \begin{equation*}
    \tu{\tau(1)}(x_{\tau(1)})\tu{\tau(2)}(x_{\tau(2)})\cdots \tu{\tau(j-1)}(x_{\tau(j-1)})\in\bR^{1\times r_j},
 \end{equation*}
 and
 \begin{equation*}
     \tu{\tau(j)}(x_{\tau(j}))\tu{\tau(j+1)}(x_{\tau(j+1)})\cdots \tu{\tau(d)}(x_{\tau(d)})\in\bR^{r_j\times 1}.
 \end{equation*}
 We can know that $\Tilde{M}$ is the product of a matrix of size $(n_{\tau(1)}n_{\tau(2)}\cdots n_{\tau(j-1)})\times r_j$ and a matrix of size $r_j \times (n_{\tau(j)}n_{\tau(j+1)}\cdots n_{\tau(d)})$ and hence that $\rk(\Tilde{M})\leq r_j$. Therefore, it holds that $\rk(M_{i_1,(i_2,i_3)})\leq r_j$ since $M_{i_1,(i_2,i_3)}$ is a submatrix of $\Tilde{M}$. Then \eqref{rank_M1_upperbd_TT} holds by taking the minimality of $j\in\{j_1+1,j_1+2,\dots,j_2\}$.
 \end{proof}

 \begin{lemma}\label{lem:rk_lowbd_TT}
 Suppose that $i_s = \tau(j_s)$ for $s=1,2,3$, where $1\leq j_1<j_2<j_3\leq d$, and that there exist $R_1$ and $R_2$ satisfying
 \begin{equation*}
     R_1\leq \min_{j_1+1\leq j \leq j_2} r_j,\ R_2\leq \min_{j_2+1\leq j \leq j_3} r_j, 
 \end{equation*}
 and
 \begin{equation*}
     n_{i_1}\geq R_1,\  n_{i_2}\geq R_1 R_2,\ n_{i_3}\geq R_2.
 \end{equation*}
 Then there exists a measure-zero subset $\Omega\in \calU_{\vec{r},\vec{n},\tau}^d$, such that for any $\tu{}\in \calU_{\vec{r},\vec{n},\tau}^d\backslash \Omega$, it holds that
 \begin{equation*}
     \rk(M_{i_2,(i_3,i_1)})\geq R_1 R_2.
 \end{equation*}
 \end{lemma}
 
 \begin{proof}
 Similar to the proof of Lemma~\ref{lem:rk_lowbd_TR}, one only needs to construct a $\tu{}\in\calU_{\vec{r},\vec{n},\tau}^d$ such that the rank of the matrix $M_1$ associated with $\tu{}$ is at least $R_1 R_2$. Set
 \begin{equation*}
 \begin{split}
     & \tu{i_1}(x_{i_1}) =\begin{cases} e_{1,r_{j_1}} e_{q,r_{j_1 + 1}}^\top,&\text{if}\ x_{i_1} = q,\  1\leq q\leq R_1,\\
     0_{r_{j_1}\times r_{j_1+1}}, &\text{if}\  R_1 < x_{i_1} \leq n_{i_1},\end{cases} \\
     & \tu{i_2}(x_{i_2}) =\begin{cases} e_{p,r_{j_2}} e_{q,r_{j_2 + 1}}^\top,&\text{if}\ x_{i_2} = (p - 1) R_2 +q,\  1\leq p\leq R_1,\ 1\leq q \leq R_2,\\
     0_{r_{j_2}\times r_{j_2+1}}, &\text{if}\  R_1 R_2 < x_{i_2} \leq n_{i_2},\end{cases} \\
     & \tu{i_3}(x_{i_3}) =\begin{cases} e_{p,r_{j_3}} e_{1,r_{j_3 + 1}}^\top,&\text{if}\ x_{i_3} = p,\  1\leq p\leq R_2,\\
     0_{r_{j_3}\times r_{j_3+1}}, &\text{if}\  R_2 < x_{i_3} \leq n_{i_3},\end{cases}
 \end{split}
 \end{equation*}
 and
 \begin{equation*}
 \tu{\tau(j)}(y_{\tau(j)}) =
 \begin{cases}
     e_{1, r_j} e_{1, r_{j+1}}^\top, &\text{if}\ 1\leq j<j_1, \\
     \sum_{p=1}^{R_1} e_{p, r_j} e_{p, r_{j+1}}^\top, &\text{if}\  j_1<j<j_2, \\
     \sum_{p=1}^{R_2} e_{p, r_j} e_{p, r_{j+1}}^\top, &\text{if}\  j_2<j<j_3, \\
     e_{1, r_j} e_{1, r_{j+1}}^\top, &\text{if}\  j_3<j\leq d.
 \end{cases}
 \end{equation*}
 Then we can compute for $x_{i_1} = q_1$, $x_{i_2} = (p_2 - 1) R_2 +q_2$, $x_{i_3} = p_3$, where $1\leq q_1,p_2\leq R_1$, $1\leq q_2,p_3 \leq R_2$, that
 \begin{align*}
     & M_{i_2,(i_3,i_1)}(x_{i_2},(x_{i_3},x_{i_1})) \\
     = & M_{i_2,(i_3,i_1)}\left(x_{\tau(j_2)},(x_{\tau(j_3)},x_{\tau(j_1)})\right) \\
     = & \prod_{j=1}^{j_1-1} \tu{\tau(j)}(y_{\tau(j)}) \cdot \tu{\tau(j_1)}(y_{\tau(j_1)})\cdot\prod_{j=j_1+1}^{j_2-1} \tu{\tau(j)}(y_{\tau(j)}) \cdot \tu{\tau(j_2)}(y_{\tau(j_2)}) \\
     &\quad \cdot\prod_{j=j_2+1}^{j_3-1} \tu{\tau(j)}(y_{\tau(j)}) \cdot \tu{\tau(j_3)}(y_{\tau(j_3)})\cdot\prod_{j=j_3+1}^d \tu{\tau(j)}(y_{\tau(j)}) \\
     =&\prod_{j=1}^{j_1-1} e_{1, r_j} e_{1, r_{j+1}}^\top\cdot e_{1,r_{j_1}} e_{q_1,r_{j_1 + 1}}^\top\cdot \prod_{j_1+1}^{j_2-1} \sum_{p=1}^{R_1} e_{p, r_j} e_{p, r_{j+1}}^\top\cdot e_{p_2,r_{j_2}} e_{q_2,r_{j_2 + 1}}^\top\\
     &\quad \cdot \prod_{j=j_2+1}^{j_3-1} \sum_{p=1}^{R_2} e_{p, r_j} e_{p, r_{j+1}}^\top \cdot e_{p_3,r_{j_3}} e_{1,r_{j_3 + 1}}^\top\cdot \prod_{j=j_3+1}^{d} e_{1, r_j} e_{1, r_{j+1}}^\top\\
     =& \left(e_{q_1,r_{j_1 + 1}}^\top\cdot \sum_{p=1}^{R_1} e_{p, r_{j_1+1}} e_{p, r_{j_2}}^\top\cdot e_{p_2,r_{j_2}}\right) \cdot\left( e_{q_2,r_{j_2 + 1}}^\top\cdot \sum_{p=1}^{R_2} e_{p, r_{j_2+1}} e_{p, r_{j_3}}^\top \cdot e_{p_3,r_{j_3}}\right)\\
     =& \delta_{q_1,p_2}\delta_{q_2,p_3},
 \end{align*}
 which combined with the fact that $M_{i_2,(i_3,i_1)}(x_{i_2},(x_{i_3},x_{i_1}))=0$ as long as $x_{i_1} \leq R_1$, $x_{i_2}\leq R_1 R_2$, and $x_{i_3}\leq R_2$ do not hold simultaneously, yields that $\rk(M_{i_2,(i_3,i_1)})=R_1 R_2$. So the proof is completed.
 \end{proof}
 
 Then one can prove Theorem~\ref{thm:TT}.
 
 \begin{proof}[Proof of Theorem~\ref{thm:TT}]
 Since Assumption~\ref{asp:exist_R_TT} holds, for any $(i_1,i_2,i_3)$ of the correct order with respect to $\tau$, by Lemma~\ref{lem:rk_upbd_TT} and Lemma~\ref{lem:rk_lowbd_TT}, we have
 \begin{equation*}
     \rk(M_{i_2,(i_3,i_1)}) \geq R^2 >\max\left\{\rk(M_{i_1,(i_2,i_3)}), \rk(M_{i_3,(i_1,i_2)})\right\},
 \end{equation*}
 as long as $\tu{}$ is not in some measure-zero set of $\calU_{\vec{r},\vec{n},\tau}^d$. Thus, one can immediately conclude the correctness of Algorithm~\ref{alg: 3index}.
\end{proof}

\subsection{Proof of Theorem~\ref{thm:robust_TT} and Corollary~\ref{cor:robust_TT}} 
The proof is similar to that of Theorem~\ref{thm:robust_TR} and Corollary~\ref{cor:robust_TR}.

\begin{proof}[Proof of Theorem~\ref{thm:robust_TT}]
Consider three different indices $i_1,i_2,i_3\in\{1,2,\dots,d\}$ and assume that $(i_1,i_2,i_3)$ is of the correct order with respect to $\tau$. Denote $E$ as the noise matrix in observing $M_{i_1,(i_2,i_3)}$, $M_{i_2,(i_3,i_1)}$, or $M_{i_3,(i_1,i_2)}$. With size at most $n_{\max}\times n_{\max}^2$, the matrix $E$ by Theorem~\ref{thm:large_sing_value_gauss} satisfies that
\begin{equation*}
    \mathbb{P}\left(\sigma_{\max}(E)\geq \frac{\sigma}{2}\right)\leq \mathbb{P}\left(\sigma_{\max}(E) > (n_{\max} + \sqrt{n_{\max}} + t) \sigma_e\right)\leq 2 e^{-t^2/2},
\end{equation*}
where $t = \frac{T}{2}- n_{\max} - \sqrt{n_{\max}}$. According to Theorem~\ref{thm:weyl_ineq} (Weyl's inequality) and Assumption~\ref{assump:uniform_SingValue_TT}, the probability that Algorithm~\ref{alg: 3index} obtains $\sigma_{R^2}(M_{i_1,(i_2,i_3)})<\sigma/2$, $\sigma_{R^2}(M_{i_2,(i_1,i_3)})>\sigma/2$, and $\sigma_{R^2}(M_{i_3,(i_1,i_2)})<\sigma/2$ is at least
\begin{equation*}
    1- 6\cdot \exp\left(-\frac{t^2}{2}\right) = 1 - 6\cdot \exp\left(-\frac{1}{8}(T-2 n_{\max} - 2\sqrt{n_{\max}})^2\right). 
\end{equation*}
The proof is hence completed.
\end{proof}

\begin{proof}[Proof of Corollary~\ref{cor:robust_TT}]
Noticing that Algorithm~\ref{alg: 3index} is called for $O(d\log d)$ times when implementing Algorithm~\ref{alg: whole_train}, one can therefore conclude Corollary~\ref{cor:robust_TT} from Theorem~\ref{thm:robust_TT}.
\end{proof}

\end{document}